\documentclass[11pt]{amsart}

\usepackage[a4paper,hmargin=3.5cm,vmargin=3.5cm]{geometry}
\usepackage{amsfonts,amssymb,amscd,amstext}
\usepackage{graphicx}
\usepackage[dvips]{epsfig}

\usepackage{fancyhdr}
\usepackage{times}

\usepackage{enumerate}
\usepackage{titlesec}
\usepackage{mathrsfs}

\def\Ccal{\mathcal{C}}

\def\Ncal{\mathcal{N}}
\def\Mcal{\mathcal{M}}

\def\Hcal{\mathcal{H}}
\def\Bcal{\mathcal{B}}
\def\Lcal{\mathcal{L}}
\def\Dcal{\mathcal{D}}

\def\Scal{\mathcal{S}}

\def\Fr{\mathrm{Fr}}
\def\dist{\mathrm{dist}}

\def\Tsf{\mathsf{T}}
\def\asf{\mathsf{a}}

\def\Bscr{\mathscr{B}}
\def\I{\mathfrak{I}}

\def\c{\mathbb{C}}
\def\r{\mathbb{R}}
\def\n{\mathbb{N}}
\def\z{\mathbb{Z}}

\def\s{\mathbb{S}}

\def\Agot{\mathfrak{A}}
\def\Bgot{\mathfrak{B}}
\def\Pgot{\mathfrak{P}}
\def\Fgot{\mathfrak{F}}
\def\Ogot{\mathfrak{O}}

\def\ggot{\mathfrak{g}}
\def\pgot{\mathfrak{p}}
\def\igot{\mathfrak{i}}
\def\jgot{\mathfrak{j}}

\def\dgot{\mathfrak{d}}

\newcommand{\sub}[1]{\text{\b{$#1$}}}

\titleformat{\section}
{\filcenter\bfseries\large} {\thesection{.}}{0.2cm}{}
\titleformat{\subsection}[runin]
{\bfseries} {\thesubsection{.}}{0.15cm}{}[.]
\titleformat{\subsubsection}[runin]
{\em}{\thesubsubsection{.}}{0.15cm}{}[.]

\usepackage[up,bf]{caption}

\newtheorem{theorem}{Theorem}[section]

\newtheorem{claim}[theorem]{Claim}
\newtheorem{lemma}[theorem]{Lemma}

\newtheorem{remark}[theorem]{Remark}
\newtheorem{definition}[theorem]{Definition}

\theoremstyle{definition}

\numberwithin{equation}{section}
\numberwithin{figure}{section}

\usepackage{color}

\begin{document}

\fancyhead[LO]{Complete non-orientable minimal surfaces and asymptotic behavior} 
\fancyhead[RE]{A. Alarc\'{o}n$\,$  and$\,$ F.J. L\'{o}pez} 
\fancyhead[RO,LE]{\thepage} 

\thispagestyle{empty}

\thispagestyle{empty}

\vspace*{1cm}
\begin{center}
{\bf\LARGE Complete non-orientable minimal surfaces in $\r^3$\\ and asymptotic behavior}

\vspace*{0.5cm}

{\large\bf Antonio Alarc\'{o}n$\;$ and$\;$ Francisco J.\ L\'{o}pez}

\end{center}

\footnote[0]{\vspace*{-0.4cm}

\noindent A.\ Alarc\'{o}n, F.\ J.\ L\'{o}pez

\noindent Departamento de Geometr\'{\i}a y Topolog\'{\i}a, Universidad de Granada, E-18071 Granada, Spain.

\noindent e-mail: {\tt alarcon@ugr.es}, {\tt fjlopez@ugr.es}

\vspace*{0.1cm}

\noindent A.\ Alarc\'{o}n is supported by Vicerrectorado de Pol\'{i}tica Cient\'{i}fica e Investigaci\'{o}n de la Universidad de Granada.

\noindent A.\ Alarc\'{o}n and F.\ J.\ L\'{o}pez's research is partially supported by MCYT-FEDER research project MTM2011-22547 and Junta de Andaluc\'{\i}a Grant P09-FQM-5088.
}


\begin{quote}
{\small
\noindent {\bf Abstract}$\;$ In this paper we give new existence results for complete non-orientable minimal surfaces in $\r^3$ with prescribed topology and asymptotic behavior.

\vspace*{0.1cm}



}
\end{quote}

\vspace*{0.25cm}


\section{Introduction}\label{sec:intro}

Non-orientable surfaces appear quite naturally in the origin itself of Minimal Surface theory and present a rich and interesting geometry.

This is part of a series of papers devoted to exploit the {\em Runge-Mergelyan type approximation theorem for non-orientable minimal surfaces}, furnished by the authors in \cite{AL-RMnon}, in order to construct non-orientable minimal surfaces in $\r^3$ with involved geometries.

The first main result of this paper concerns complete non-orientable minimal surfaces in $\r^3$ spanning a finite collection of closed curves.

\begin{theorem}\label{th:intro-compact}
Let $\Scal_0$ be an open non-orientable smooth surface with finite topology.

Then there exist a relatively compact domain $\Scal$ in $\Scal_0$ and a continuous map $X\colon \overline\Scal\to\r^3$ such that $\Scal$ is homeomorphic to $\Scal_0$, the Hausdorff dimension of $X(\overline \Scal\setminus\Scal)$ equals $1$, and the restriction $X|_{\Scal}\colon \Scal\to\r^3$ is a complete minimal immersion. 

Furthermore, the flux of the immersion $X|_{\Scal}$ can be arbitrarily prescribed.
\end{theorem}

A map $X$ as those given by Theorem \ref{th:intro-compact} is said to be a {\em non-orientable compact complete minimal immersion}. The ones in the above theorem are the first examples of such immersions in the literature. 

We point out that our method does not give control over the topology of $\overline\Scal\setminus \Scal$. In particular, we can not ensure that $\overline\Scal\setminus \Scal$ consists of a finite collection of Jordan curves (see Remark \ref{rem:huuhuesr} below).

In the orientable setting, compact complete minimal immersions of the disc into $\r^3$ were constructed by Mart\'in and Nadirashvili \cite{MN-Jordan}; examples with arbitrary finite topology were given later by Alarc\'on \cite{Alarcon-Compact}.  Other related results can be found in \cite{AN,Alarcon-Compact2}. The construction methods used in \cite{MN-Jordan,Alarcon-Compact} are refinements of Nadirashvili's technique for constructing complete bounded minimal surfaces in $\r^3$; see \cite{Nadirashvili}. In the context of null holomorphic curves in $\c^3$ (i.e., holomorphic immersions from Riemann surfaces into $\c^3$ whose real and imaginary parts are conformal minimal immersions), Alarc\'on and L\'opez \cite{AL-Israel} gave compact complete examples with any given finite topological type. Their method, which relies on the {\em Runge-Mergelyan  theorem for null curves in $\c^3$} (see \cite{AL-proper}), is the inspiration of our proof.

Compact complete minimal surfaces in $\r^3$ are interesting objects since they lie in the intersection of two well known topics on minimal surface theory: the {\em Plateau problem} (dealing with the existence of {\em compact} minimal surfaces spanning a given family of closed curves in $\r^3$) and the {\em Calabi-Yau problem} (concerning the existence of {\em complete} minimal surfaces in bounded regions of $\r^3$). See the already cited sources and references therein for a more detailed discussion.

The second main result of this paper regards with complete non-orientable minimal surfaces in $\r^3$ properly projecting into planar convex domains.

\begin{theorem}\label{th:intro-proper}
Let $\Scal$ be an open non-orientable smooth surface (possibly with infinite topology) and let $\Dcal\subset\r^2$ be a convex domain.

Then there exists a complete minimal immersion $X=(X_1,X_2,X_3)\colon \Scal\to\r^3$ such that $(X_1,X_2)(\Scal)\subset \Dcal$ and $(X_1,X_2)\colon \Scal\to \Dcal$ is a proper map.

Furthermore, the flux of the immersion $X$ can be arbitrarily prescribed.
\end{theorem}

The problem of whether there exist minimal surfaces in $\r^3$ with {\em hyperbolic} conformal structure and {\em properly projecting} into $\r^2$ was proposed by Schoen and Yau \cite{SchoenYau-harmonic}. (Recall that an open Riemann surface is said to be hyperbolic if it carries non-constant negative subharmonic functions; otherwise it is said to be {\em parabolic}.) This question was settled in the affirmative by the authors in both the orientable and the non-orientable settings \cite{AL-proper,AL-conjugada,AL-RMnon}. More specifically,  such surfaces with any given conformal structure and flux map were provided. Theorem \ref{th:intro-proper} shows that the corresponding result for {\em complete} non-orientable surfaces and {\em convex domains} of $\r^2$ holds as well; cf.\ \cite{AL-CY} for the orientable case. On the other hand, Ferrer, Mart\'in, and Meeks \cite{FerrerMartinMeeks} provided complete non-orientable minimal surfaces, with arbitrary topology, properly immersed in any given convex domain $\Omega$ of $\r^3$; however, if the domain $\Omega$ is a right cylinder over a convex domain $\Dcal$ of $\r^2$, their method does not provide any information about the projection of the surface into $\Dcal$. 

Although our techniques are inspired by those already developed in the orientable setting (cf.\ \cite{AL-Israel,AL-CY}), the non-orientable character of the surfaces requires a much more careful discussion. Indeed, every non-orientable minimal surface $S$ in $\r^3$ can be represented by a triple $(\Ncal,\I,X)$, where $\Ncal$ is an open Riemann surface, $\I\colon\Ncal\to\Ncal$ is an antiholomorphic involution without fixed points, and $X\colon\Ncal\to\r^3$ is a conformal minimal immersion satisfying 
\begin{equation}\label{eq:compa}
X\circ \I=X
\end{equation}
and $S=X(\Ncal)$; see Subsec.\ \ref{sec:minimal} for details. The moduli space of open Riemann surfaces admitting an antiholomorphic involution without fixed points is real analytic and rather subtle; as a matter of fact this condition implies not only topological restrictions on the surfaces but also conformal ones. Moreover, the required compatibility \eqref{eq:compa} with respect to the antiholomorphic involution makes the construction of non-orientable minimal surfaces a much more involved problem. In order to overcome these difficulties, we exploit the Runge-Mergelyan theorem for non-orientable minimal surfaces \cite{AL-RMnon} (see Theorem \ref{th:Mergelyan} below). This flexible tool enables us to obtain the examples in Theorems \ref{th:intro-compact} and \ref{th:intro-proper} as limit of sequences of compact non-orientable minimal surfaces (with non-empty boundary), considerably simplifying the construction methods in Sec.\ \ref{sec:theorem} and Sec.\ \ref{sec:theorem2}.


\subsection*{Outline of the paper} In Sec.\ \ref{sec:prelim} we introduce the background and notation about Riemann surfaces, non-orientable minimal surfaces, and convex domains, that will be needed throughout the paper. In particular, we state the Runge-Mergelyan theorem for non-orientable minimal surfaces \cite{AL-RMnon}; see Theorem \ref{th:Mergelyan}. With this approximation result in hand, Theorems \ref{th:intro-compact} and \ref{th:intro-proper} are proved in Sec.\ \ref{sec:theorem} and Sec.\ \ref{sec:theorem2}, respectively; see the more general Theorems \ref{th:compact} and \ref{th:proper}.


%
%


\section{Preliminaries}\label{sec:prelim}

We denote by $\|\cdot\|$, $\langle\cdot,\cdot\rangle$, and $\dist(\cdot,\cdot)$ the Euclidean norm, metric, and distance in $\r^n$, $n\in\n$. Given a compact topological space $K$ and a continuous map $f\colon K\to \r^n,$ we denote by 
\begin{equation}\label{eq:norma00}
\|f\|_{0,K}:= \max \big\{ \|f(p)\|\colon p\in K \big\}
\end{equation}
the maximum norm of $f$ on $K.$ The corresponding space of continuous functions on $K$ will be endowed with the $\Ccal^0$ topology associated to $\|\cdot\|_{0,K}.$

Given a topological surface $N,$ we denote by $b N$ the (possibly non-connected) $1$-dimensional topological manifold determined by its boundary points. Open connected subsets of $N\setminus b N$ will be called {\em domains}. Proper connected topological subspaces of $N\setminus bN$ being compact surfaces with boundary will be said {\em regions}. For any subset $A \subset N,$ we denote by $A^\circ$, $\overline{A}$, and $\Fr A=\overline{A}\setminus A^\circ$, the interior, the closure, and the topological frontier of $A$ in $N$, respectively.  Given subsets $A$, $B$ of $N$, we say that $A \Subset B$ if $\overline{A}$ is compact and $\overline{A}\subset B^\circ$.


\subsection{Riemann surfaces and non-orientability}\label{sec:riemann}

A Riemann surface $\Ncal$ is said {\em open} if it is non-compact and $b \Ncal =\emptyset.$ For such $\Ncal$, we denote by $\partial$ the global complex operator given by $\partial|_U=\frac{\partial}{\partial z} dz$ for any conformal chart $(U,z)$ on $\Ncal.$ 

Riemann surfaces are orientable; the conformal structure of a Riemann surface induces a (positive) orientation on it. The natural notion of {\em non-orientable Riemann surface} is described as follows; see \cite{Meeks-8pi,AL-RMnon} for a detailed exposition of this issue.

\begin{definition}\label{def:non}
By a {\em non-orientable Riemann surface} we mean an orbit space $\Ncal/\I$; where $\Ncal$ is an open Riemann surface and $\I\colon\Ncal\to\Ncal$ is an antiholomorphic involution without fixed points. Therefore, a non-orientable Riemann surface is identified with an open Riemann surface $\Ncal$ equipped with an antiholomorphic involution $\I$ without fixed points.

In this setting, $\Ncal$ is the two-sheets conformal orientable cover of $\Ncal/\I$. We denote by $\pi\colon \Ncal \to \Ncal/\I$ the natural projection.
Further, $\Ncal$ carries conformal Riemannian metrics $\sigma_\Ncal^2$ such that $\I^*(\sigma_\Ncal^2)=\sigma_\Ncal^2.$
\end{definition}

From now on in this section, let $\Ncal$, $\I$, $\pi$, and $\sigma_\Ncal^2$ be as in Def.\ \ref{def:non}.

\begin{definition}\label{def:invariant}
A subset $A\subset \Ncal$ is said to be $\I$-invariant if  $\I(A)=A$.

For an $\I$-invariant set  $A\subset\Ncal$, a map $f\colon A\to\r^n$, $n\in\n$, is said to be {\em $\I$-invariant} if $f\circ \I|_A=f$.
\end{definition}

Let $\Gamma\subset\Ncal$ be an $\I$-invariant subset consisting of finitely many pairwise disjoint smooth Jordan curves $\gamma_j,$ $j=1,\ldots,k.$ For any $\epsilon>0$ we denote by
\[
\Tsf_\epsilon(\Gamma):=\{P \in \Ncal\colon \dist_{\sigma_\Ncal^2}(P,\Gamma)<\epsilon\};
\]
where $\dist_{\sigma_\Ncal^2}$ means Riemannian distance in $(\Ncal,\sigma_\Ncal^2).$ Notice that $\Tsf_\epsilon(\Gamma)\subset\Ncal$ is an $\I$-invariant set. If $\epsilon$ is small enough, the exponential map
\[
F\colon \Gamma \times [-\epsilon,\epsilon] \to \overline{\Tsf_\epsilon(\Gamma)},\enskip 
F(P,t)=\exp_P(t\, {\sf n}(P)),
\]
is a diffeomorphism and $\Tsf_\epsilon(\Gamma)=F(\Gamma \times (-\epsilon,\epsilon))$; where $\sf n$ is an $\I$-invariant normal field along $\Gamma$ in $(\Ncal,\sigma_\Ncal^2).$ In this setting, $\Tsf_\epsilon(\Gamma)$ is said to be a {\em metric tubular neighborhood} of $\Gamma$ (of radius $\epsilon$). Furthermore, if $\pi_\Gamma\colon \Gamma \times (-\epsilon,\epsilon) \to \Gamma$ denotes the projection $\pi_\Gamma(P,t)=P,$ we denote by 
\[
\Pgot_\Gamma\colon \Tsf_\epsilon(\Gamma) \to \Gamma,\enskip \Pgot_\Gamma (Q):=\pi_\Gamma(F^{-1}(Q)),
\]
the natural orthogonal projection. Since $\Gamma$, ${\sf n}$, and $\sigma_\Ncal^2$ are $\I$-invariant, then 
\[
F\circ(\I\times {\rm Id})|_{\Gamma\times[-\epsilon,\epsilon]}=\I\circ F \enskip\text{and}\enskip \Pgot_\Gamma\circ\I|_{\Tsf_\epsilon(\Gamma)}=\I\circ \Pgot_\Gamma.
\]

\begin{definition}\label{def:bordered}  
A domain $U\subset\Ncal$ is said to be {\em bordered} if it is the interior of a compact Riemann surface $\overline{U}\subset\Ncal$ with smooth boundary. In this case $b\overline{U}=\Fr U\subset\Ncal$ consists of finitely many closed Jordan curves.

If $\Ncal$ is of finite topology, we denote by $\Bscr_\I(\Ncal)$ the family of $\I$-invariant bordered domains $U\Subset\Ncal$ such that $\Ncal$ is a topological tubular neighborhood of $U$. (The latter means that $\Ncal\setminus \overline{U}$ has no relatively compact connected components and consists of finitely many open annuli.)
\end{definition}


\subsection{Non-orientable minimal surfaces} \label{sec:minimal}

In this subsection we describe the Weierstrass representation for non-orientable minimal surfaces (see \cite{Meeks-8pi}), and introduce some notation.

An $\I$-invariant conformal minimal immersion $X\colon \Ncal\to\r^3$ induces a conformal minimal immersion $\sub X \colon \Ncal/\I\to\r^3$, satisfying $X=\sub X \circ \pi$. In this sense, $X(\Ncal)$ is an immersed non-orientable minimal surface in $\r^3$.  Conversely, any immersed non-orientable minimal surface in $\r^3$ comes in this way. 


Let $X=(X_j)_{j=1,2,3}\colon \Ncal \to \r^3$ be an $\I$-invariant conformal minimal immersion. Denote by $\phi_j=\partial X_j,$ $j=1,2,3,$ and $\Phi=\partial X\equiv (\phi_j)_{j=1,2,3}.$  The $1$-forms  $\phi_j$ are holomorphic,  have no real periods, and satisfy
\begin{equation}\label{eq:conformal}
\sum_{j=1}^3 \phi_j^2=0
\end{equation}
and
\begin{equation}\label{eq:Wdata-non}
\I^*\Phi=\overline{\Phi}.
\end{equation}
The intrinsic metric in $\Ncal$ is given by
\begin{equation}\label{eq:metric}
ds^2=\sum_{j=1}^3 |\phi_j|^2;
\end{equation}
hence 
\begin{equation}\label{eq:immersion}
\text{$\sum_{j=1}^3 |\phi_j|^2$ vanishes nowhere on $\Ncal$}.
\end{equation}
The triple $\Phi$ is said to be the {\em Weierstrass representation} of $X$. 

Conversely, any vectorial holomorphic $1$-form $\Phi=(\phi_j)_{j=1,2,3}$ on $\Ncal$ without real periods, enjoying \eqref{eq:conformal}, \eqref{eq:Wdata-non}, and \eqref{eq:immersion}, determines an $\I$-invariant conformal minimal immersion $X\colon \Ncal\to\r^3$ by the expression
\[
X=\Re\int\Phi,
\]
where $\Re$ means real part. Cf. \cite{Meeks-8pi}.

The following notation will be required later on.

\begin{definition}\label{def:M(A)}
For any $\I$-invariant subset $A\subset \Ncal,$  we denote by $\Mcal_\I(A)$ the space of $\I$-invariant conformal minimal immersions of $\I$-invariant open domains  $W\subset \Ncal$, containing $A$, into $\r^3.$ 
\end{definition}

Given an $\I$-invariant connected subset $A\subset\Ncal$ and $X\in\Mcal_\I(A)$, we denote by $\dist_X$ the distance in $A$ associated to the intrinsic metric induced by $X$; that is, 
\[
\dist_X(P,Q)=\inf \big\{\ell(X(\gamma))\colon \gamma\text{ arc in $A$ connecting $P$ and $Q$}\big\},
\]
where $\ell$ means Euclidean length in $\r^3$.

\begin{definition} \label{def:conor}
Let $A$ be a subset of $\Ncal,$ let $X$ be a conformal minimal immersion from an open subset containing $A$ into $\r^3$, and let $\gamma(s)$ be an arc-length parameterized curve in $A.$ The {\em conormal vector field}  of $X$ along $\gamma$ is the unique unitary tangent vector field  $\mu$ of $X$ along $\gamma$ such that $\{d X(\gamma'(s)),\mu(s)\}$ is a positive basis for all $s.$
If in addition $\gamma$ is closed, then the {\em flux} $\pgot_X(\gamma)\in\r^3$ of $X$ along $\gamma$ is given by $\int_\gamma \mu(s) ds.$ 
\end{definition}

If $\gamma$ is closed, it is easy to check that
\[
\pgot_X(\gamma)=\Im\int_{\gamma} \partial X
\]
(here $\Im$ means imaginary part), and that the {\em flux map} $\pgot_X\colon \Hcal_1(A,\z)\to \r^3$ is a group morphism. Furthermore, if $A$ and $X$ are $\I$-invariant, then the flux map $\pgot_X\colon \Hcal_1(A,\z)\to\r^3$ of $X$ satisfies 
\begin{equation}\label{eq:flux-non}
\pgot_X(\I_*(\gamma))=-\pgot_X(\gamma)\enskip \forall \gamma\in \Hcal_1(A,\z);
\end{equation}
recall that $\I\colon\Ncal\to\Ncal$ reverses the orientation.


\subsection{$\I$-admissible sets and $\I$-invariant generalized minimal immersions}\label{sec:generalized}

In this subsection we introduce the necessary notation for a well understanding of the Runge-Mergelyan type approximation result for non-orientable minimal surfaces, given by the authors in \cite{AL-RMnon}, which is the key tool in the present paper; see Theorem \ref{th:Mergelyan} below.

\begin{remark}\label{rem:inicio}
From now on in the paper, $\Ncal,$ $\I$, and $\pi$ will be as in Def.\ \ref{def:non}. We also fix a conformal Riemannian metric $\sigma_\Ncal^2$ on $\Ncal$ such that $\I^*(\sigma_\Ncal^2)=\sigma_\Ncal^2.$
\end{remark}

 A compact Jordan arc in $\Ncal$ is said to be analytical (smooth, continuous, etc.) if it is contained in an open analytical (smooth, continuous, etc.) Jordan arc in $\Ncal.$

\begin{definition}\label{def:admi}
A (possibly non-connected) $\I$-invariant compact subset $S\subset \Ncal$ is said to be $\I$-admissible (in $\Ncal$) if it meets the following requirements (see Fig.\ \ref{fig:admi}):
\begin{enumerate}[\rm (a)]
\item $S$ is Runge (in $\Ncal$); i.e., $\Ncal\setminus S$ has no relatively compact connected components in $\Ncal$.
\item $R_S:=\overline{S^\circ}$ is non-empty and consists of a finite collection of pairwise disjoint compact regions in $\Ncal$ with   $\Ccal^0$ boundary.
\item $C_S:=\overline{S\setminus R_S}$ consists of a finite collection of pairwise disjoint analytical Jordan arcs.
\item Any component $\alpha$ of $C_S$  with an endpoint  $P\in R_S$ admits an analytical extension $\beta$ in $\Ncal$ such that the unique component of $\beta\setminus\alpha$ with endpoint $P$ lies in $R_S$.
\end{enumerate}
\end{definition}
\begin{figure}[ht]
    \begin{center}
    \scalebox{0.22}{\includegraphics{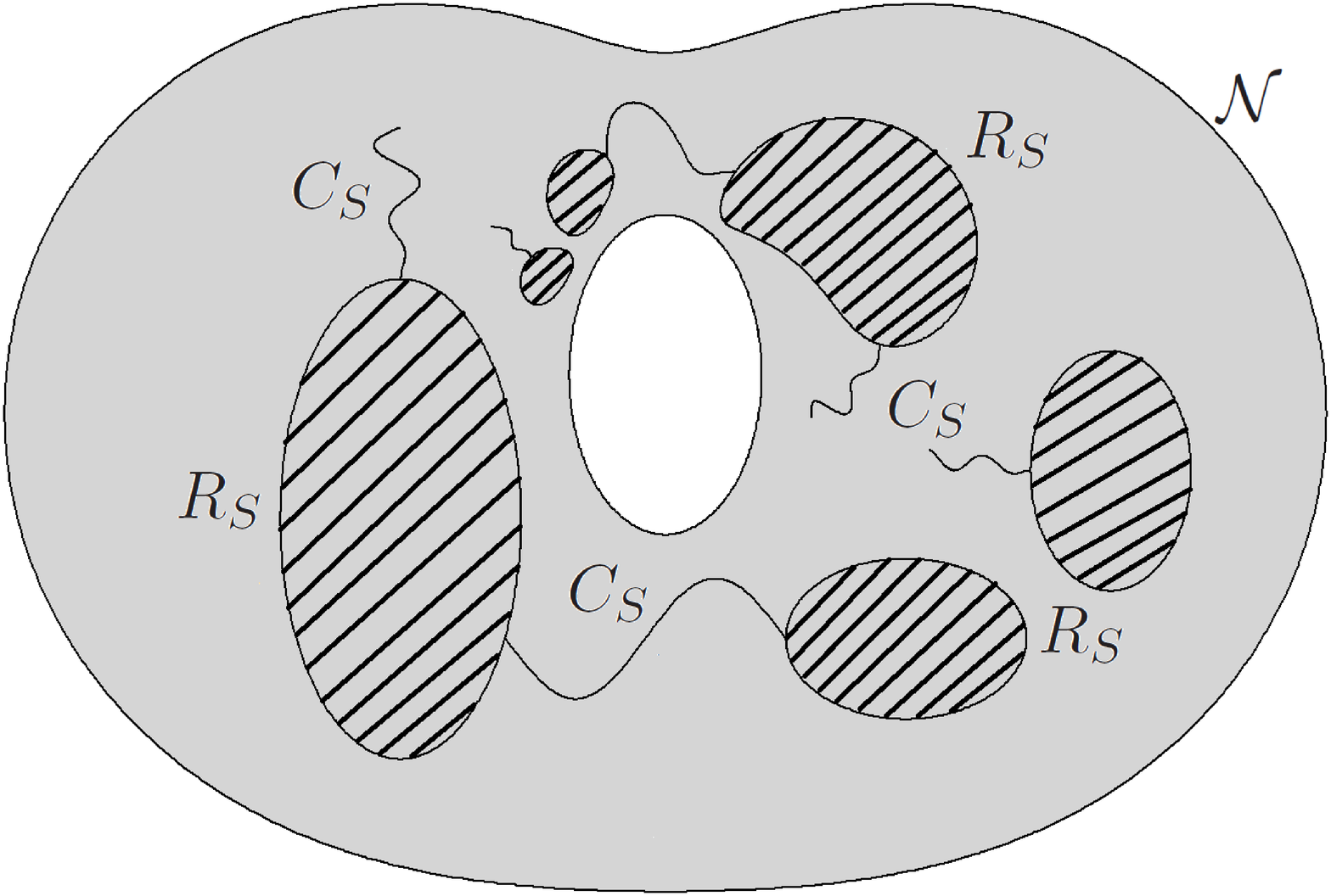}}
        \end{center}
\caption{An $\I$-admissible set $S\subset\Ncal.$}
\label{fig:admi}
\end{figure}

An $\I$-invariant compact subset $S\subset \Ncal$ enjoying {\rm (b)}, {\rm (c)}, and {\rm (d)},  is $\I$-admissible if and only if $i_*\colon \Hcal_1(S,\z) \to \Hcal_1(\Ncal,\z)$ is a monomorphism; where $\Hcal_1(\cdot,\z)$ means first homology group, $i\colon S \to \Ncal$ denotes the inclusion map,  and $i_*$ is the induced group morphism. 

From now on in this section, let $S\subset\Ncal$ be an $\I$-admissible set.

\begin{definition}\label{def:gene-min}
We say that an $\I$-invariant map $X\colon S\to\r^3$ is an {\em $\I$-invariant generalized minimal immersion}, and write $X\in \Mcal_{\ggot,\I}(S)$, if it meets the following requirements:
\begin{itemize}
\item $X|_{R_S} \in \Mcal_\I(R_S)$ (see Def.\ \ref{def:M(A)}); hence it extends as an $\I$-invariant conformal minimal immersion $X_0$ to an open domain $V$ in $\Ncal$ containing $R_S$.

\item For any component $\alpha$ of $C_S$ and any open analytical Jordan arc $\beta$ in $\Ncal$ containing $\alpha,$  $X|_\alpha$ is a regular curve admitting a smooth extension $X_\beta$ to $\beta$ such that $X_\beta|_{V \cap \beta}=X_0|_{V \cap \beta}$.
\end{itemize}

\end{definition}

Notice that $X|_{S}\in \Mcal_{\ggot,\I}(S)$  for all $X \in \Mcal_\I(S).$ 

Let $X\in \Mcal_{\ggot,\I}(S)$, and let $\varpi$ be  a smooth 3-dimensional real 1-form on $C_S$. This means that $\varpi=(\varpi_j)_{j=1,2,3}$, where $\varpi_j$ is a real smooth 1-form on $C_S$, $j=1,2,3$. For any $\alpha\subset C_S$ we write  $\varpi|_\alpha=\varpi(\alpha(s))ds$,   where $s$ is the arc-length parameter of $X\circ \alpha$. By definition, $\varpi$ is said to be a {\em mark} along $C_S$ with respect to $X$ if  for any arc $\alpha\subset C_S$ the following conditions hold:
\begin{itemize}
\item  $\varpi(\alpha(s))\in \r^3$ is a smooth unitary vector field along $\alpha$ orthogonal  to $(X\circ \alpha)'(s)$.
\item $\varpi$ extends smoothly  to any open analytical  arc $\beta$ in $\Ncal$ containing $\alpha$.
\item $\varpi(\beta(s))$ is unitary, orthogonal to $(X\circ \beta)'(s)$, and tangent to $X(R_S)$ at $\beta(s)$ for any $\beta(s)\in \beta \cap R_S$, where as above $s$ is the arc-length parameter of $(X\circ \beta)(s)$.
\end{itemize}

Let ${\mathfrak{n}}\colon R_S\to\s^2$ denote the Gauss map of the (oriented) conformal minimal immersion $X|_{R_S}$.  The mark $\varpi$ is said to be {\em orientable} with respect to $X$ if the orientations at the two endpoints of each arc in $C_S$ agree, that is to say, if there exists $\delta \in \{-1,1\}$ such that for any  regular embedded curve $\alpha \subset S$ and arc-length parametrization $(X\circ \alpha)(s)$ of $X\circ \alpha$, 
\[
\text{$(X\circ\alpha)'({s_0})\times  \varpi(\alpha({s_0}))= \delta {\mathfrak{n}}(\alpha({s_0}))\quad$ for all ${s_0}\in \alpha^{-1}(C_S\cap R_S)$.}
\]

Orientable marks along $C_S$ with respect to $X$ always exist since $\Ncal$ is orientable. An orientable mark $\varpi$ with respect to $X$  is said to be {\em positively oriented} if  $\delta=1$. Obviously, if $\varpi$ is orientable with respect to $X$ then either $\varpi$ or $-\varpi$ is positively oriented. 

In the sequel we will only consider orientable marks.

If $\varpi$ is a positively oriented mark along $C_S$ with respect to $X$, we denote by ${\mathfrak{n}}_\varpi\colon S\to \s^2\subset \r^3$ the map given by ${\mathfrak{n}}_\varpi|_{R_S}={\mathfrak{n}}$ and $({\mathfrak{n}}_\varpi \circ \alpha)(s):=(X\circ \alpha)'(s)\times \varpi(\alpha(s))$, where $\alpha$ is any component of
$C_S$ and $s$ is any  arc-length parameter of $X\circ\alpha$.  By definition,  ${\mathfrak{n}}_\varpi$ is said to be the {\em (generalized) Gauss map} of $X$ associated to the orientable mark $\varpi$.

\begin{definition}\label{def:marked}
We denote by  $\Mcal_{\ggot,\I}^*(S)$ the space of marked immersions $X_\varpi:=(X,\varpi),$ where $X \in \Mcal_{\ggot,\I}(S)$ and $\varpi$ is a positively oriented mark along $C_S$ with respect to $X$ satisfying the following properties:
\begin{itemize}
\item $\I^*(\varpi)=-\varpi$, or equivalently,
\begin{equation}\label{eq:con-non}
{\mathfrak{n}}_\varpi \circ \I=-{\mathfrak{n}}_\varpi.
\end{equation}
\item If ${\rm st}\colon\s^2\to\overline\c$ is the stereographic projection, the function ${\rm st}\circ {\mathfrak{n}}_\varpi\colon S\to\overline\c$, which is holomorphic on an open neighborhood of $R_S$, is smooth in an analogous way to Def.\ \ref{def:gene-min}.
\end{itemize}
\end{definition}

A $1$-form $\theta$ on $S$ is said to be of {\em type $(1,0)$} if for any conformal chart $(U,z)$ in $\Ncal,$ $\theta|_{U \cap S}=h(z) dz$ holds for some function $h\colon U \cap S \to \overline{\c}.$
Finite sequences $\Theta=(\theta_1,\ldots,\theta_n),$ where $\theta_j$ is a $(1,0)$-type $1$-form for all $j\in\{1,\ldots,n\},$ are said to be $n$-dimensional vectorial $(1,0)$-forms on $S.$   The space of continuous $n$-dimensional $(1,0)$-forms on $S$ will be endowed with the $\Ccal^0$ topology induced by the norm 
\begin{equation} \label{eq:norma0}
\|\Theta\|_{0,S}:=\big\|\frac{\Theta}{\sigma_\Ncal}\big\|_{0,S}=\max_{S} \big\{ \big(\sum_{j=1}^n |\frac{\theta_j}{\sigma_\Ncal}|^2\big)^{1/2}\big\} .
\end{equation}

\begin{definition}\label{def:wei-gen}
For every $X_\varpi \in \Mcal_{\ggot,\I}^*(S),$ we denote by $\partial X_\varpi$ the complex vectorial $1$-form on $S$ given by  
\[
\partial X_\varpi|_{R_S}=\partial (X|_{R_S}),\enskip \partial X_\varpi(\alpha'(s))= dX
(\alpha'(s)) + \imath \varpi(s);
\]
where $\imath=\sqrt{-1}$,
\begin{itemize}
\item $dX$ denotes the vectorial $1$-form of type $(1,0)$ on $C_S$  given by
\[
dX|_{\alpha \cap U}=(X \circ \alpha)'(x)dz|_{\alpha \cap U},
\]
for any component $\alpha$ of $C_S,$ where $(U,z=x+\imath y)$ is any conformal chart on $\Ncal$ satisfying that $z(\alpha \cap U)\subset \r$ (the existence of such a conformal chart is guaranteed by the analyticity of $\alpha$), and

\item $s$ is the arc-length parameter of $X|_\alpha$  for
which $\{dX (\alpha'(s_i)), \varpi(s_i)\}$ are positive, where $s_1$ and $s_2$ are the values of $s$ for which  $\alpha(s) \in b R_S.$
\end{itemize}
\end{definition}

In the setting of Def.\ \ref{def:wei-gen}, writing $\partial X_\varpi=(\widehat{\phi}_j)_{j=1,2,3}$, it follows that
\begin{itemize}
\item $\sum_{j=1}^3 \widehat \phi_j^2$ vanishes everywhere on $S$, 
\item $\sum_{j=1}^3 |\widehat \phi_j|^2$ vanishes nowhere on $S$, and
\item $\I^* (\partial X_\varpi)=\overline{\partial X_\varpi}$.
\end{itemize}
For these reasons the triple $\partial X_\varpi$ is said the {\em generalized Weierstrass representation} of $X_\varpi$.

For $F \in \Mcal_\I(S)$, we denote by $\varpi_F$ the conormal field of $F$ along  $C_S.$ Notice that $\varpi_F$ satisfies \eqref{eq:con-non} and $(\partial F)|_S=\partial F_{\varpi_F}$; where $F_{\varpi_F}:=(F|_S,\varpi_F) \in \Mcal_{\ggot,\I}^*(S).$ 

The space $\Mcal_{\ggot,\I}^*(S)$ is naturally endowed with the following $\Ccal^1$ topology:

\begin{definition}\label{def:C1}
Let $W$ be an $\I$-invariant open domain in $\Ncal$ containing $S$.
\begin{itemize}
\item Given $X_{\varpi},$ $Y_{\xi}\in\Mcal_{\ggot,\I}^*(S),$ we set
\[
\|X_{\varpi}-Y_{\xi}\|_{1,S}:=\|X-Y\|_{0,S}+\big\|\partial X_{\varpi}-\partial Y_{\xi}\big\|_{0,S} \enskip\text{(see \eqref{eq:norma00} and \eqref{eq:norma0})}.
\] 

\item Given $F,$ $G\in \Mcal_\I(S),$ we set 
\[
\|F-X_{\varpi}\|_{1,S}:=\|F_{\varpi_F}-X_{\varpi}\|_{1,S}\enskip \text{and} \enskip \|F-G\|_{1,S}:=\|F_{\varpi_F}-G_{\varpi_G}\|_{1,S}.
\]

\item We will say that $X_\varpi \in \Mcal_{\ggot,\I}^*(S)$ {\em can be  approximated in the $\Ccal^1$ topology on $S$ by $\I$-invariant conformal minimal immersions in $\Mcal_\I(W)$} if for any $\epsilon>0$ there exists  $Y\in  \Mcal_\I(W)$ such that $\|Y-X_\varpi\|_{1,S}<\epsilon.$
\end{itemize}
\end{definition} 

If $X_\varpi\in\Mcal_{\ggot,\I}^*(S),$ then the group homomorphism  
\begin{equation}\label{eq:fluxgene}
\pgot_{X_\varpi}\colon \Hcal_1(S,\z) \to \r^3, \enskip \pgot_{X_\varpi}(\gamma)=\Im \int_\gamma \partial X_\varpi,
\end{equation}
is said to be the {\em generalized flux map} of $X_\varpi.$ Notice that $\pgot_{X_\varpi}$  satisfies \eqref{eq:flux-non}. Obviously,  $\pgot_{X_{\varpi_Y}}=\pgot_Y|_{\Hcal_1(S,\z)}$ provided that  $X=Y|_S$ for some $Y\in\Mcal_\I(S).$

The following Runge-Mergelyan type approximation result for non-orientable minimal surfaces plays a fundamental role in this paper.

\begin{theorem}[\cite{AL-RMnon}]\label{th:Mergelyan}
Let $S\subset \Ncal$ be an $\I$-admissible subset (see Def.\ \ref{def:admi}), let $X_\varpi \in \Mcal_{\ggot,\I}^*(S)$ (see Def.\ \ref{def:marked}), and let $\pgot\colon \Hcal_1(\Ncal,\z)\to\r^3$ be a group homomorphism such that 
\begin{itemize}
\item $\pgot(\I_*(\gamma))=-\pgot(\gamma)$ for all $\gamma\in \Hcal_1(\Ncal,\z)$, and  
\item $\pgot|_{\Hcal_1(S,\z)}$ equals the generalized flux map $\pgot_{X_\varpi}$ of $X_\varpi$.
\end{itemize}
 Write $X_\varpi=(X=(X_j)_{j=1,2,3},\varpi)$, $\partial X_\varpi=(\phi_j)_{j=1,2,3}$, and $\pgot=(\pgot_j)_{j=1,2,3}$.

Then the following assertions hold:

\begin{enumerate}[\rm (I)]
\item $X_\varpi$ can be approximated in the $\Ccal^1$ topology on $S$ (see Def.\ \ref{def:C1}) by $\I$-invariant conformal minimal immersions $Y\in\Mcal_{\I}(\Ncal)$ with flux map $\pgot_Y=\pgot$.

\item If $\phi_3$ does not vanish everywhere on $S$ and extends to $\Ncal$ as a holomorphic $1$-form without real periods, vanishing nowhere on $C_S$, and satisfying $\pgot_3(\gamma)=\Im\int_\gamma \phi_3$ for all $\gamma\in\Hcal_1(\Ncal,\z)$, then $X_\varpi$ can be approximated in the $\Ccal^1$ topology on $S$ by $\I$-invariant conformal minimal immersions $Y=(Y_j)_{j=1,2,3}\in\Mcal_{\I}(\Ncal)$ with flux map $\pgot_Y=\pgot$ and third coordinate $Y_3=X_3$.
\end{enumerate}
\end{theorem}


\subsection{Convex domains and Hausdorff distance}\label{sec:convex}

A convex domain $\Dcal\subset\r^n$, $\Dcal \neq \r^n$, $n\geq 2$, is said to be {\em regular} (resp., {\em analytic}) if its frontier $\Fr\Dcal=\overline{\Dcal}\setminus\Dcal$ is a regular (resp.,  analytic) hypersurface of $\r^n$. Moreover, $\Dcal$ is said to be {\em strictly convex} if $\Fr\Dcal$ contains no straight segments.

For any couple of compact subsets $K$ and $O$ in $\r^n$,
the {\em Hausdorff distance} between $K$ and $O$ is given by
\[
\dgot^{\rm H}(K,O):=\max\Big\{ \sup_{x\in K} \inf_{y\in O} \|x-y\|\;,\;
\sup_{y\in K} \inf_{x\in O} \|x-y\|\Big\}.
\]

A sequence $\{K^j\}_{j \in \n}$ of (possibly unbounded) closed subsets of $\r^n$ is said to converge in the {\em Hausdorff topology} to a closed subset $K^0$ of $\r^n$ if
$\{K^j\cap B\}_{j \in \n}\to K^0 \cap B$ in the Hausdorff distance for any closed Euclidean ball $B \subset \r^n.$ If $K^j\Subset K^{j+1} \subset K^0$ for all $j\in\n$ and $\{K^j\}_{j \in \n} \to K^0$ in the Hausdorff topology, then we write $\{K^j\}_{j \in \n} \nearrow K^0.$ 


\begin{theorem}[\cite{Minkowski,MeeksYau}]\label{th:mink}
Let $\Bcal\subset\r^n$ be a (possibly neither bounded nor regular) convex domain. 
Then there exists a sequence $\{\Dcal^j\}_{j\in\n}$ of   bounded strictly convex analytic domains in $\r^n$ with $\{\overline{\Dcal^j}\}_{j \in \n} \nearrow \overline{\Bcal}.$
\end{theorem}


\section{Compact complete non-orientable minimal immersions}\label{sec:theorem}

In this section we prove Theorem \ref{th:intro-compact}, which is a particular instance of Theorem \ref{th:compact} below.

Recall that we have fixed $\Ncal,$ $\I$, $\pi$, and $\sigma_\Ncal^2$ as in Def.\ \ref{def:non}; see Remark \ref{rem:inicio}.
Throughout this section $\Ncal$ is assumed to be of finite topology. See Def.\ \ref{def:bordered} and Def.\ {def:M(A)} for notation.

\begin{theorem}\label{th:compact}
Let $U\in \Bscr_\I(\Ncal)$, let $K\subset U$ be a compact set, and let $X\in\Mcal_\I(\overline{U})$.

Then, for any $\epsilon>0$ there exists an $\I$-invariant domain $M$ in $\Ncal$ and a continuous $\I$-invariant map $Y\colon \overline{M}\to\r^3$, enjoying the following properties:
\begin{itemize}
\item $K\subset M\Subset U$ and the inclusion map $\Hcal_1(M,\z)\hookrightarrow \Hcal_1(\Ncal,\z)$ is an isomorphism.
\item $Y|_M\colon M\to\r^3$ is a complete conformal $\I$-invariant minimal immersion.
\item $\|Y-X\|_{1,K}<\epsilon$.
\item $\|Y-X\|_{0,\overline{M}}<\epsilon$ and the Hausdorff distance $\dgot^{\rm H}\big(X(b\overline{U}), Y(\Fr M)\big)<\epsilon$.
\item If $\Fr M=\Gamma\cup \I(\Gamma)$ with $\Gamma\cap\I(\Gamma)=\emptyset$, then $Y|_{\Gamma}\colon \Gamma\to\r^3$ is injective.
\item The Hausdorff dimension of $Y(\Fr M)$ equals $1$.
\item The flux map $\pgot_Y$ of $Y$ equals the one $\pgot_X$ of $X$.
\end{itemize}
\end{theorem}

\begin{remark}\label{rem:huuhuesr}
Theorem \ref{th:compact} does not insure that $M$ is a bordered domain in $\Ncal$. In particular, we can not guarantee that $\Fr M$ consists of a finite collection of pairwise disjoint Jordan curves; the same happens to $Y(\Fr M)$.
\end{remark}



As usual in this kind of constructions, the map $Y$ in Theorem \ref{th:compact} will be constructed in a recursive procedure; the key tool in this process is Lemma \ref{lem:compact2} below. Most of the technical arguments in the proof of Lemma \ref{lem:compact2} are contained in the following

\begin{lemma}\label{lem:compact}
Let $U\in\Bscr_\I(\Ncal)$, let $K\subset U$ be a compact set, let $\Tsf\Subset \Ncal\setminus K$ be a metric tubular neighborhood of $b\overline{U}$ in $\Ncal$, and let $\Pgot\colon \Tsf\to b\overline{U}$ be the orthogonal projection. Let $X\in \Mcal_\I(\overline{U})$, let $\Fgot\colon b\overline{U}\to\r^3$ be an $\I$-invariant analytical map, and let $\mu>0$ such that
\begin{equation}\label{eq:lemma0}
\|X-\Fgot\|_{0,b\overline{U}}<\mu.
\end{equation}

Then, for any $\rho>0$ and any $\epsilon>0$ there exist $V\in\Bscr_\I(\Ncal)$ and $Y\in\Mcal_\I(\overline{V})$ enjoying the following properties:
\begin{enumerate}[\rm (i)]
\item $K\subset V\Subset U$ and $b\overline{V}\subset \Tsf$.
\item If $b\overline V=\Gamma\cup\I(\Gamma)$ with $\Gamma\cap\I(\Gamma)=\emptyset$, then $Y|_{\Gamma}\colon \Gamma\to \r^3$ is an embedding.
\item $\dist_Y(K,b\overline{V})>\rho.$
\item $\|Y-X\|_{1,K}<\epsilon$.
\item $\|Y-\Fgot\circ\Pgot\|_{0,b\overline{V}}<\sqrt{4\rho^2+\mu^2}+\epsilon$.
\item The flux map $\pgot_Y$ of $Y$ equals the one $\pgot_X$ of $X$.
\end{enumerate}
\end{lemma}

\begin{proof}
Let $\epsilon_0>0$.

Since $U\in\Bscr_\I(\Ncal)$, then
\[
b\overline{U}=\cup_{i=1}^\igot (\beta_i\cup\I(\beta_i)),
\]
where $\{\beta_i\}_{i=1}^\igot$ are pairwise disjoint smooth Jordan curves with $\beta_i\cap \I(\beta_j)=\emptyset$ for all $i,j\in\{1,\ldots,\igot\}$. Denote by $\beta=\cup_{i=1}^\igot \beta_i$. Obviously, $b\overline{U}=\beta\cup\I(\beta)$ and $\beta\cap\I(\beta)=\emptyset$.

For any $P\in \beta$ we choose a simply connected open neighborhood $O_P$ of $P$ in $\overline{U}\cap \Tsf$ meeting the following requirements:
\begin{enumerate}[{\rm ({A}1)}]
\item $\Pgot(Q)\in O_P\cap \beta_i$ for all $Q\in O_P$ and $P\in \beta_i$, $i=1,\ldots,\igot$.
\item $\max\big\{\|X(Q_1)-X(Q_2)\|\,,\,\|\Fgot(\Pgot(Q_1))-\Fgot(\Pgot(Q_2))\|\big\}<\epsilon_0,$ for all $\{Q_1,Q_2\}\subset O_P$ and $P\in \beta_i$, $i=1,\ldots,\igot$.
\item $\|X-\Fgot\circ\Pgot\|_{0,O_P}<\mu$ for all $P\in \beta$.
\item $O_P\cap\I(O_P)=\emptyset$ for all $P\in \beta$.
\end{enumerate}
Notice that {\rm (A3)} is ensured by hyphotesis \eqref{eq:lemma0}, provided that $O_P$ is chosen small enough. To guarantee {\rm (A2)}, just take $O_P$ sufficiently small and use the continuity of $X$, $\Fgot$, and $\Pgot$.

Set 
\[
\Ogot=\{O_P\colon P\in \beta\},
\]
and observe that $\Ogot\cap \beta:=\{O_P\cap \beta \colon P\in \beta\}$ is an open covering of $\beta.$ Choose $M\in\Bscr_\I(\Ncal)$ satisfying that
\begin{equation}\label{eq:M} 
K\subset M \Subset U, \enskip \overline{U}\setminus M \subset \cup_{P\in \beta} (O_P\cup \I(O_P)), \text{and}\enskip \Pgot|_{b\overline M}\colon b\overline M\to b\overline U \text{ is one to one}.
\end{equation}
For instance, one can take $M$ as the complement in $\overline U$ of a sufficiently small metric tubular neighborhood of $b\overline U$.

Since $U,M\in\Bscr_\I(\Ncal)$ and \eqref{eq:M}, then  
\begin{equation}\label{eq:inT}
\overline U\setminus M=\cup_{i=1}^\igot (A_i\cup\I(A_i))\subset \Tsf;
\end{equation}
where $\{A_i\}_{i=1}^\igot$ are pairwise disjoint compact annuli with $A_i\cap \I(A_j)=\emptyset$ for all $i,j$, and $\beta_i\subset A_i$, $i=1,\ldots,\igot$. Denote $\alpha_i=A_i\cap b\overline{M}$, $i=1,\ldots,\igot$, $\alpha=\cup_{i=1}^\igot \alpha_i$, and $A=\cup_{i=1}^\igot A_i$. Obviously, $b\overline{M}=\alpha\cup\I(\alpha)$, $\alpha\cap\I(\alpha)=\emptyset$, $\overline{U}\setminus M=A\cup\I(A)$, $A\cap\I(A)=\emptyset$, and $A\subset\cup_{P\in\beta}O_P$.
In particular, $\Ogot\cap A:=\{O_P\cap A\colon P\in \beta\}$ is an open covering of $A$. 

Denote by $\z_n$ the additive cyclic group of integers modulus $n$, $n\in\n$. Since $\Ogot\cap A$ is an open covering of the compact set $\alpha$ in $A$, then there exist $\jgot\in\n,$ $\jgot\geq 3,$ and a family of compact Jordan arcs $\{\alpha_{i,j}\colon (i,j)\in I=\{1,\dots,\igot\}\times\z_\jgot\}$ meeting the following requirements:
\begin{enumerate}[{\rm ({B}1)}]
\item $\cup_{j\in\z_\jgot} \alpha_{i,j}=\alpha_i$ for all $i\in\{1,\dots,\igot\}$.
\item $\alpha_{i,j}$ and $\alpha_{i,j+1}$ have a common endpoint $Q_{i,j}$ and are otherwise disjoint for all $(i,j)\in I$.
\item $\alpha_{i,j}\cap \alpha_{i,k}=\emptyset$ for all $(i,j)\in I$ and $k\in \z_\jgot\setminus\{j,j+1\}$.
\item $\alpha_{i,j}\cup\alpha_{i,j+1}\subset O_{i,j}\in\Ogot$ for all $(i,j)\in I$.
\end{enumerate}

Up to suitably trimming the $O_{i,j}$'s we can further assume that
\begin{enumerate}[{\rm ({B}1)}]
\item[\rm (B5)] $O_{i,j-1}\cap O_{i,j}\cap O_{i,j+1}$ is connected for any $(i,j)\in I$. Observe that $Q_{i,j}\in O_{i,j-1}\cap O_{i,j}\cap O_{i,j+1}\neq\emptyset$ by {\rm (B2)} and {\rm (B4)}.
\end{enumerate}

For any $(i,j)\in I$, choose a point $P_{i,j}\in O_{i,j-1}\cap O_{i,j}$ and set 
\begin{equation}\label{eq:eij}
e_{i,j}^3:=\left\{
\begin{array}{ll}
\frac{X(P_{i,j})-\Fgot(\Pgot(P_{i,j}))}{\|X(P_{i,j})-\Fgot(\Pgot(P_{i,j}))\|} & \text{if }X(P_{i,j})-\Fgot(\Pgot(P_{i,j}))\neq 0
\vspace*{0.2cm}\\
\text{any vector in $\s^2$} & \text{if }X(P_{i,j})-\Fgot(\Pgot(P_{i,j}))= 0.\\
\end{array}\right.
\end{equation}
Observe that $e_{i,j}^3\in\s^2$ and the orthogonal projection of $X(P_{i,j})-\Fgot(\Pgot(P_{i,j}))$ into the orthogonal complement of $e_{i,j}^3$ in $\r^3$ equals $0$.

For any $(i,j)\in I$ choose $\{e_{i,j}^1,e_{i,j}^2,e_{i,j}^3\}$ an orthonormal basis of $\r^3$, and denote by $B_{i,j}\in{\rm O}(3,\r)$ the orthogonal matrix of change of coordinates in $\r^3$ from the canonical basis to the basis $\{e_{i,j}^1,e_{i,j}^2,e_{i,j}^3\}$; i.e., 
\begin{equation}\label{eq:Bij}
B_{i,j}=\Big({e_{i,j}^1}^T\,,\,{e_{i,j}^2}^T\,,\,{e_{i,j}^3}^T\Big)^{-1},
\end{equation}
where $\cdot^T$ means transpose.

Let $\{r_{i,j}\colon (i,j)\in I\}$ be a family of pairwise disjoint analytical compact Jordan arcs in $A$ meeting the following requirements:
\begin{enumerate}[\rm ({C}1)]
\item $r_{i,j}\subset O_{i,j-1}\cap O_{i,j}\cap O_{i,j+1}$ for all $(i,j)\in I$.
\item $r_{i,j}$ has initial point $Q_{i,j}$, final point $\Pgot(Q_{i,j})$, and it is otherwise disjoint from $\alpha_i\cup\beta_i,$  for all $(i,j)\in I$.
\item The set $S:=\overline{M}\cup\big(\cup_{(i,j)\in I} (r_{i,j}\cup\I(r_{i,j}))\big)\subset \overline U\subset \Ncal$ is $\I$-admissible in the sense of Def.\ \ref{def:admi}. 
\end{enumerate} 
See Fig.\ \ref{fig:compact}. For instance, one can take $r_{i,j}=\Pgot^{-1}(\Pgot(Q_{i,j}))\cap (\overline U\setminus M)$ for all $(i,j)\in I$; see \eqref{eq:M}. 

Properties {\rm (C1)} and {\rm (C2)} are possible thanks to \eqref{eq:M}, {\rm (A1)}, {\rm (B2)}, {\rm (B4)}, and {\rm (B5)}. Notice that $r_{i,j}\cap\I(r_{i,j})=\emptyset$ for all $(i,j)\in I$; see {\rm (A4)} and {\rm (C1)}.

The first main step in the proof of Lemma \ref{lem:compact} consists of deforming $X$ over $\cup_{(i,j)\in I}(r_{i,j}\cup\I(r_{i,j}))$. To do this we first extend $X|_{\overline M}$ to an $\I$-invariant generalized minimal immersion $\widehat X\in \Mcal_{\ggot,\I}(S)$ (see Def.\ \ref{def:gene-min}) enjoying the following properties:
\begin{enumerate}[\rm (a)]
\item $\widehat X|_{\overline{M}}=X$.
\item $\|\widehat X(P)-\widehat X(Q)\|<\epsilon_0$ for all $\{P,Q\}\subset r_{i,j}$, for all $(i,j)\in I$.
\item If $\Upsilon\subset r_{i,j}$ is a Borel measurable subset, then
\begin{eqnarray*}
\min\big\{\ell\big(\pi_{i,j}(\widehat X(\Upsilon))\big)\,,\,\ell\big(\pi_{i,j+1}(\widehat X(\Upsilon))\big)\big\} & + & \\
\min\big\{\ell\big(\pi_{i,j}(\widehat X(r_{i,j}\setminus\Upsilon))\big)\,,\,\ell\big(\pi_{i,j+1}(\widehat X(r_{i,j}\setminus\Upsilon))\big)\big\} & > &2\rho,
\end{eqnarray*}
where $\ell$ denotes Euclidean length in $\r^3$ and 
\begin{equation}\label{eq:piij}
\text{$\pi_{i,j}\colon \r^3\to {\rm span}\{e_{i,j}^3\}\leq\r^3$ the orthogonal projection,}
\end{equation}
for all $(i,j)\in I$.
\end{enumerate}

To construct $\widehat X$ we first define it over each arc $r_{i,j}$ to be highly oscillating in the direction of both $e_{i,j}^3$ and $e_{i,j+1}^3$ (property {\rm (c)}), but with small diameter in $\r^3$ (property {\rm (b)}). We then define $\widehat X$ over each arc $\I(r_{i,j})$ just to be $\I$-invariant.

Theorem \ref{th:Mergelyan} applied to any marked immersion $\widehat X_\varpi=(\widehat X,\varpi)\in\Mcal_{\ggot,\I}^*(S)$ and $\pgot=\pgot_{\widehat X_\varpi}=\pgot_X\colon\Hcal_1(\Ncal,\z)\to\r^3$ (recall that $M,U\in \Bscr_\I(\Ncal)$ and see property {\rm (a)} of $\widehat X$), furnishes an $\I$-invariant conformal minimal immersion $F\in\Mcal_\I(\overline{U})$ satisfying:
\begin{enumerate}[\rm ({D}1)]
\item $\|F-X\|_{1,\overline{M}}<\epsilon_0$.
\item $\|F-X\|_{0,S}<2\epsilon_0$.
\item If $\Upsilon\subset r_{i,j}$ is a Borel measurable subset, then
\begin{eqnarray*}
\min\big\{\ell\big(\pi_{i,j}(F(\Upsilon))\big)\,,\,\ell\big(\pi_{i,j+1}(F(\Upsilon))\big)\big\} & + & \\
\min\big\{\ell\big(\pi_{i,j}(F(r_{i,j}\setminus\Upsilon))\big)\,,\,\ell\big(\pi_{i,j+1}(F(r_{i,j}\setminus\Upsilon))\big)\big\} & > &2\rho,
\end{eqnarray*}
for all $(i,j)\in I$; see \eqref{eq:piij}.
\item The flux map $\pgot_F$ of $F$ equals the one $\pgot_X$ of $X$.
\end{enumerate}
Take into account properties {\rm (a)}, {\rm (b)}, and {\rm (c)} of $\widehat X$. For {\rm (D2)}, recall that $\widehat X(Q_{i,j})=X(Q_{i,j})$ for all $(i,j)\in I$ and use properties {\rm (C1)}, {\rm (A2)}, and {\rm (b)}.

By continuity of $F$, there exists $W\in\Bscr_\I(\Ncal)$ such that:
\begin{enumerate}[\rm ({E}1)]
\item $K\subset M\Subset W\subset U$.
\item $S\subset \overline{W}$ and $S\cap b\overline{W}=\cup_{(i,j)\in I} \{\Pgot(Q_{i,j}),\I(\Pgot(Q_{i,j}))\}$; recall {\rm (C2)}.
\item $\overline{W}\setminus M=\cup_{(i,j)\in I} \big(\big(\widetilde{\alpha}_{i,j}\cup \I(\widetilde{\alpha}_{i,j})\big)\cup  \big(\widetilde{r}_{i,j}\cup \I(\widetilde{r}_{i,j})\big)\big)$, where $\widetilde{\alpha}_{i,j}$ and $\widetilde{r}_{i,j}$ are simply connected compact neighborhoods of $\alpha_{i,j}$ and $r_{i,j}$, respectively, in $\overline{W}\setminus M$, such that
\begin{itemize}
\item $\widetilde{\alpha}_{i,j}\cap \widetilde{r}_{i,k}=\emptyset$ for all $(i,j)\in I$ and $k\in \z_\jgot\setminus \{j-1,j\}$.
\item $\widetilde{r}_{i,j}\cap \widetilde{r}_{i,k}=\emptyset$ for all $(i,j)\in I$ and $k\in \z_\jgot\setminus \{j\}$.
\end{itemize}
\item $\|F-X\|_{1,\overline{M}}<\epsilon_0$.
\item $\|F-X\|_{0,\overline W}<2\epsilon_0$.
\item Denote by $\gamma_{i,j}$ the piece of $b\overline{W}$ connecting $\Pgot(Q_{i,j-1})$ and $\Pgot(Q_{i,j}),$ and containing $\Pgot(Q_{i,k})$ for no $k\in\z_\jgot\setminus\{j-1,j\}$. Then $\gamma_{i,j}$ is split into three compact connected sub-arcs $\gamma_{i,j}^{-1}\subset \Fr(\widetilde{r}_{i,j-1})$, $\Gamma_{i,j}$, and $\gamma_{i,j}^{1}\subset \Fr(\widetilde{r}_{i,j})$, such that $\gamma_{i,j}^{-1}\cap \gamma_{i,j}^{1}=\emptyset$, $\Gamma_{i,j}$ has a common point with $\gamma_{i,j}^{-1}$ and a common point with $\gamma_{i,j}^{1}$, and the following assertion holds:

For any arc $\sigma\subset \widetilde r_{i,j}$ connecting $\widetilde{\alpha}_{i,j}\cup \widetilde{\alpha}_{i,j+1}$ and $\gamma_{i,j}^1\cup \gamma_{i,j+1}^{-1}$, if $\sigma=\sigma_0\cup\sigma_1$, where $\sigma_k$ is a collection of subarcs of $\sigma$ contained in the closure of the connected component of $\widetilde r_{i,j}\setminus r_{i,j}$ intersecting $\alpha_{i,j+k}$, $k=0,1$, then
\[
\ell\big(\pi_{i,j}(F(\sigma_0))\big)+\ell\big(\pi_{i,j+1}(F(\sigma_1))\big)  >2\rho \enskip\text{for all $(i,j)\in I$.}
\]
\item $(\partial F B_{i,j}^T)_3$ vanishes nowhere on $b\overline W$ for all $(i,j)\in I$; here $(\cdot)_3$ means third coordinate in $\r^3$ and $B_{i,j}$ is given by \eqref{eq:Bij}.
\end{enumerate}
See Fig.\ \ref{fig:compact}.
\begin{figure}[ht]
    \begin{center}
    \scalebox{0.45}{\includegraphics{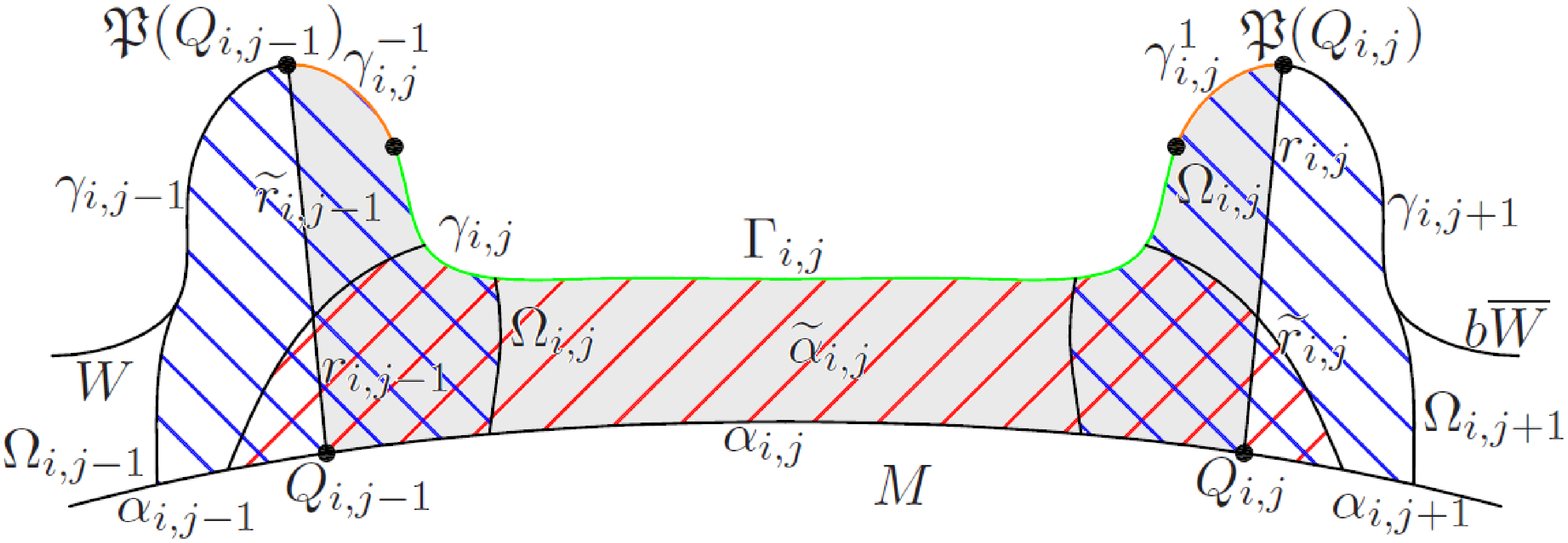}}
        \end{center}
\caption{$\overline W\setminus M$}
\label{fig:compact}
\end{figure}

Denote by $\Omega_{i,j}$ the closure of the connected component of $\overline W\setminus S$ bounded by $r_{i,j-1}$, $\alpha_{i,j}$, $\alpha_{i,j+1}$, and $\gamma_{i,j}$, for all $(i,j)\in I$. Observe that $\Omega_{i,j}$ is a closed disc for all $(i,j)\in I$ and $\overline W= \overline M\cup (\cup_{(i,j)\in I} (\Omega_{i,j}\cup \I(\Omega_{i,j})))$. See Fig.\ \ref{fig:compact}.

Let $\eta\colon \{1,\ldots,\igot\jgot\}\to\{1,\ldots,\igot\}\times \z_\jgot$ be the bijection $\eta(k)=(E(\frac{k-1}{\jgot})+1,k-1),$ where $E(\cdot)$ means integer part. 

The second main step in the proof of Lemma \ref{lem:compact} consists of deforming $F$ over each disc $\Omega_{\eta(k)}$, $k\in \{1,\ldots,\igot\jgot\}$. For that, let us recursively construct a sequence $\{F_0=F, F_1,\ldots, F_{\igot\jgot}\}\subset \Mcal_\I(\overline U)$ satisfying the following properties for all $k\in \{0,1,\ldots,\igot\jgot\}$:
\begin{enumerate}[\rm ({F}.1$_{k}$)]
\item $\|F_k-X\|_{1,\overline M}<\epsilon_0$.

\item $\|F_k-X\|_{0, S\cup\big(\bigcup_{a=k+1}^{\igot\jgot}(\Omega_{\eta(a)}\cup \I(\Omega_{\eta(a)}))\big)}<2\epsilon_0$.

\item $\langle F_k-F_{k-1},e_{\eta(k)}^3 \rangle=0$, $k\geq 1$.

\item $\|F_k(Q)-X(Q)\|>2\rho+1,$ $\forall Q\in \Gamma_{\eta(a)}$, $\forall a\in\{1,\ldots, k\}$, $k\geq 1$.

\item For any arc $\sigma\subset \widetilde r_{\eta(a)}$ connecting $\widetilde{\alpha}_{\eta(a)}\cup \widetilde{\alpha}_{\eta(a)+(0,1)}$ and $\gamma_{\eta(a)}^1\cup \gamma_{\eta(a)+(0,1)}^{-1}$, then
\[
\ell\big(\pi_{\eta(a)}(F_k(\sigma\cap \Omega_{\eta(a)}))\big)+\ell\big(\pi_{\eta(a)+(0,1)}(F_k(\sigma\cap \Omega_{\eta(a)+(0,1)}))\big)  >2\rho\enskip \forall a\in \{1,\ldots,\igot\jgot\}.
\]
Recall that $\pi_{\eta(a)}\colon \r^3\to {\rm span}\{e_{\eta(a)}^3\}$ is the orthogonal projection; see \eqref{eq:piij}.

\item $\|F_k-F_{k-1}\|_{1,\overline{W\setminus (\Omega_{\eta(k)}\cup \I(\Omega_{\eta(k)}))}} <\epsilon_0/\igot\jgot$, $k\geq 1$.

\item $(\partial F_{k} B_{\eta(a)}^T)_3$ vanishes nowhere on $b\overline W$ for all $a\in\{1,\ldots,\igot\jgot\}$.

\item The flux map $\pgot_{F_k}$ of $F_k$ equals the one $\pgot_X$ of $X$.
\end{enumerate}

Indeed, observe that {\rm (F.1$_0$)}={\rm (E4)}, {\rm (F.2$_0$)} is implied by {\rm (E5)}, {\rm (F.5$_0$)}={\rm (E6)}, and {\rm (F.7$_0$)}={\rm (E7)}, whereas {\rm (F.3$_0$)}, {\rm (F.4$_0$)}, and {\rm (F.6$_0$)} make no sense. Finally {\rm (F.8$_0$)} follows from {\rm (D4)}. Reason by induction and assume that we already have $F_0,\ldots, F_{k-1}$, for some $k\in\{1,\ldots,\igot\jgot\}$, satisfying the corresponding properties. Let us construct $F_k$.

Denote $G=(G_1,G_2,G_3):=F_{k-1} B_{\eta(k)}^T\in \Mcal_\I(\overline U)$. Obviously, $G\in\Mcal_\I(\overline U)$ and 
\begin{equation}\label{eq:G3}
G_3=\langle F_{k-1}, e_{\eta(k)}^3\rangle
\end{equation}
(see.\ \eqref{eq:Bij}). Denote
\[
S_k:= \overline M \cup \Big(\Gamma_{\eta(k)}\cup \I(\Gamma_{\eta(k)})\Big)\cup \Big(\bigcup_{a\neq k} \Omega_{\eta(a)}\cup \I(\Omega_{\eta(a)})\Big)
\]
and observe that $S_k$ is $\I$-admissible (Def.\ \ref{def:admi}). Observe also that $S_k$ has exactly three connected components, which are $\overline M \cup \Big(\bigcup_{a\neq k} \Omega_{\eta(a)}\cup \I(\Omega_{\eta(a)})\Big)=\overline{W\setminus(\Omega_{\eta(k)}\cup\I(\Omega_{\eta(k)}))}$, $\Gamma_{\eta(k)}$, and $\I(\Gamma_{\eta(k)})$; see {\rm (E3)}.

Extend $G|_{S_k\setminus (\Gamma_{\eta(k)}\cup \I(\Gamma_{\eta(k)}))}$ to an $\I$-invariant generalized minimal immersion $\widehat G=(\widehat G_1,\widehat G_2,\widehat G_3)\in\Mcal_{\ggot,\I}(S_k)$ (Def. \ref{def:gene-min}) such that
$
\widehat G_3= G_3|_{S_k}
$
and
\begin{equation}\label{eq:hatG}
\|\widehat G(Q)- X(Q)B_{\eta(k)}^T\|>2\rho+1\enskip \text{for all }Q\in \Gamma_{\eta(k)}.
\end{equation}
Take for instance $\widehat G|_{\Gamma_{\eta(k)}}=(x_0,y_0,0)+G|_{\Gamma_{\eta(k)}}$ for any constant $(x_0,y_0)\in\r^2$ with sufficiently large norm. Then define $\widehat G|_{\I(\Gamma_{\eta(k)})}$ to be $\I$-invariant.

In view of {\rm (F.7$_{k-1}$)}, assertion {\rm (II)} in Theorem \ref{th:Mergelyan} applied to any marked immersion $\widehat G_\varpi=(\widehat G,\varpi)\in\Mcal_{\ggot,\I}^*(S_k)$ such that $(\partial \widehat G_\varpi)_3=(\partial G_3)|_{S_k}$ and  $\pgot=\pgot_{\widehat G_\varpi}=\pgot_G\colon\Hcal_1(\Ncal,\z)\to\r^3$, provides an $\I$-invariant conformal minimal immersion $\widetilde G=(\widetilde G_1,\widetilde G_2,\widetilde G_3)\in\Mcal_\I(\overline U)$ such that 
\begin{equation}\label{eq:tildeG}
\widetilde G_3=G_3
\end{equation}
and
\begin{equation}\label{eq:tildeG2}
\text{the flux map $\pgot_{\widetilde G}$ of $\widetilde G$ equals the one $\pgot_G$ of $G$.}
\end{equation}
Furthermore, if the approximation of $\widehat G$ by $\widetilde G$ in $S_k$ is close enough, then $F_k:=\widetilde G (B_{\eta(k)}^T)^{-1}\in\Mcal_\I(\overline U)$ meets conditions $({\rm F}.1_k)$--$({\rm F}.7_k)$.

Indeed, observe that \eqref{eq:Bij} implies that $\widetilde G_3=\langle F_k, e_{\eta(k)}^3\rangle$; hence \eqref{eq:tildeG} and \eqref{eq:G3} ensure {\rm (F.3$_k$)}. Since $\widehat G$ agrees with $G$ on $S_k\setminus (\Gamma_{\eta(k)}\cup \I(\Gamma_{\eta(k)}))=\overline{W\setminus(\Omega_{\eta(k)}\cup\I(\Omega_{\eta(k)}))}\supset S \cup\big(\bigcup_{a=k+1}^{\igot\jgot}\Omega_{\eta(a)}\cup \I(\Omega_{\eta(a)})\big)$, then {\rm (F.1$_{k-1}$)} and {\rm (F.2$_{k-1}$)} guarantee {\rm (F.1$_k$)} and {\rm (F.2$_k$)}, respectively, whereas {\rm (F.6$_k$)} directly follows. Likewise,  \eqref{eq:hatG} guarantees {\rm (F.4$_k$)}. Finally, since $\|F_k-F_{k-1}\|_{1,\overline{W\setminus (\Omega_{\eta(k)}\cup \I(\Omega_{\eta(k)}))}}\approx 0$ and $\pi_{\eta(k)}\circ F_k=\pi_{\eta(k)}\circ F_{k-1}$ (see {\rm (F.3$_k$)} and \eqref{eq:piij}), then {\rm (F.5$_{k-1}$)}  and {\rm (F.7$_{k-1}$)} ensure {\rm (F.5$_k$)} and {\rm (F.7$_k$)}, respectively. Finally, \eqref{eq:tildeG2} and {\rm (F.8$_{k-1}$)} give {\rm (F.8$_k$)}.

This concludes the construction of the sequence $\{F_0=F, F_1,\ldots, F_{\igot\jgot}\}\subset \Mcal_\I(\overline U)$.

Label $H:=F_{\igot\jgot}\in\Mcal_\I(\overline U)$. Let us prove that
\begin{claim}\label{cla:dist1}
${\rm dist}_H(\overline M,b\overline W)>2\rho.$
\end{claim}
\begin{proof}
Consider a connected curve $\sigma$ in $\overline U$ with initial point $Q\in \overline M$ and final point $P\in b\overline W$. It suffices to show that $\ell(H(\sigma))>2\rho$. Assume without loss of generality that $(\sigma\setminus\{P,Q\})\cap(\overline M\cup b\overline W)=\emptyset$ and, up to possibly replacing $\sigma$ by $\I(\sigma)$, that $\sigma\subset \cup_{k=1}^{\igot\jgot} \Omega_{\eta(k)}$. Let us distinguish cases. 

Assume $P\in \Gamma_{\eta(k)}$ for some $k\in\{1,\ldots,\igot\jgot\}$. Then there exists a point $Q_0\in\sigma\cap (r_{\eta(k)-(0,1)}\cup\alpha_{\eta(k)}\cup r_{\eta(k)})$ and it follows that 
\begin{eqnarray*}\label{eq:enK}
\ell(H(\sigma)) & \geq & \|H(P)-H(Q_0)\|\\
 & \geq & \|(H-X)(P)\|-\|X(P)-X(Q_0)\|-\|(H-X)(Q_0)\|
\\
 & \stackrel{\text{{\rm (F.4$_{\igot\jgot}$)},{\rm (A2)},{\rm (F.2$_{\igot\jgot}$)}}}{\geq} & 2\rho+1 - \epsilon_0 -2\epsilon_0 > 2\rho.
\end{eqnarray*}
For the latter inequality we assume from the beginning $\epsilon_0<1/3$.

Assume now that $P\in \Gamma_{\eta(k)}$ for no $k\in\{1,\ldots,\igot\jgot\}$. In this case there exists $k\in\{1,\ldots,\igot\jgot\}$ such that $P\in \gamma_{\eta(k)}^1\cup\gamma_{\eta(k)+(0,1)}^{-1}\subset\widetilde r_{\eta(k)}$. Therefore, there exists a connected sub-arc $\widehat\sigma\subset\sigma\cap \widetilde r_{\eta(k)}$ connecting $\widetilde\alpha_{\eta(k)}\cup\widetilde\alpha_{\eta(k)+(0,1)}\neq \emptyset$ and $P$. Then {\rm (F.5$_{\igot\jgot}$)} gives $\ell(H(\sigma))\geq  \ell(H(\widehat\sigma))>2\rho$. This proves the claim
\end{proof}

Let us now prove the following
\begin{claim}\label{cla:dist2}
The inequality 
\[
\|H(Q)-\Fgot(\Pgot(Q))\|<\sqrt{4\rho^2+\mu^2}+\epsilon,
\]
where $\rho$, $\mu$, $\epsilon$, and $\Fgot$ are the data given in the statement of Lemma \ref{lem:compact}, is satisfied for any $Q\in \overline W\setminus M$ such that ${\rm dist}_H(Q,\overline M)<2\rho$.
\end{claim}
\begin{proof}
Choose $Q\in \overline W\setminus M$ with ${\rm dist}_H(\overline M,Q)<2\rho$. In view of Claim \ref{cla:dist1} and up to possibly replacing without loss of generality $Q$ by $\I(Q)$, there exist $k\in\{1,\ldots,\igot\jgot\}$ and $P\in r_{\eta(k)-(0,1)}\cup\alpha_{\eta(k)}\cup r_{\eta(k)}\subset S$ such that
\begin{equation}\label{eq:2rho}
Q\in\Omega_{\eta(k)}\setminus\gamma_{\eta(k)} \quad\text{and}\quad \dist_H(P,Q)<2\rho.
\end{equation}
By Pitagoras Theorem,
\begin{multline}\label{eq:pita1}
\|H(Q)-\Fgot(\Pgot(Q))\|=
\\
 \sqrt{\langle H(Q)-\Fgot(\Pgot(Q)),e_{\eta(k)}^3\rangle^2+\|\Theta_{\eta(k)}\big(H(Q)-\Fgot(\Pgot(Q))\big)\|^2},
\end{multline}
where $\Theta_{\eta(k)}\colon\r^3\to {\rm span}\{e_{\eta(k)}^1,e_{\eta(k)}^2\}$ is the orthogonal projection (see \eqref{eq:Bij}).

On the one hand,  
\begin{multline*}
|\langle H(Q)-\Fgot(\Pgot(Q)),e_{\eta(k)}^3\rangle|\leq \|H(Q)-F_k(Q)\|
+ |\langle (F_k-F_{k-1})(Q),e_{\eta(k)}^3\rangle| 
\\
+ \|F_{k-1}(Q)-X(Q)\|
+ \|X(Q)-\Fgot(\Pgot(Q))\|;
\end{multline*}
hence, in view of {\rm (F.6$_a$)}, $a=k+1,\ldots,\igot\jgot$, {\rm (F.3$_k$)}, {\rm (F.2$_{k-1}$)}, and {\rm (A.3)},
\begin{equation}\label{eq:pita2}
|\langle H(Q)-\Fgot(\Pgot(Q)),e_{\eta(k)}\rangle| < \epsilon_0 + 0 + 2\epsilon_0 + \mu=\mu+3\epsilon_0.
\end{equation}

On the other hand, 
\begin{multline*}
\|\Theta_{\eta(k)}\big(H(Q)-\Fgot(\Pgot(Q))\big)\|\leq 
\\
\|H(Q)-H(P)\|
+ \|H(P)-X(P)\|
+ \|X(P)-X(P_{\eta(k)})\|+
\\
+ \|\Theta_{\eta(k)}\big(X(P_{\eta(k)})-\Fgot(\Pgot(P_{\eta(k)}))\big)\|
+ \|\Fgot(\Pgot(P_{\eta(k)}))-\Fgot(\Pgot(Q))\|;
\end{multline*}
hence, using \eqref{eq:2rho}, {\rm (F.2$_{\igot\jgot}$)}, {\rm (A.2)}, \eqref{eq:eij}, and again {\rm (A.2)},
\begin{equation}\label{eq:pita3}
\|\Theta_{\eta(k)}\big(H(Q)-\Fgot(\Pgot(Q))\big)\|< 2\rho+2\epsilon_0+\epsilon_0+ 0 + \epsilon_0=2\rho+4\epsilon_0.
\end{equation}
Combining \eqref{eq:pita1}, \eqref{eq:pita2}, and \eqref{eq:pita3} one obtains
\[
\|H(Q)-\Fgot(\Pgot(Q))\|<\sqrt{(\mu+3\epsilon_0)^2+(2\rho+4\epsilon_0)^2}<\sqrt{4\rho^2+\mu^2}+\epsilon
\]
as claimed (the latter inequality is satisfied provided that $\epsilon_0$ is chosen small enough from the beginning).
\end{proof}

Choose a bordered domain $V\in \Bscr_\I(\Ncal)$ satisfying that
\begin{equation}\label{eq:V1}
M\Subset V \Subset W
\end{equation}
and
\begin{equation}\label{eq:V2}
\rho<\dist_H (Q,\overline M)<2\rho\enskip \forall Q\in b\overline V.
\end{equation}
Existence of such a $V$ is guaranteed by Claim \ref{cla:dist1}. Moreover, up to a slight deformation of the domain $V$, it can be ensured in addition that
\begin{equation}\label{eq:V3}
H|_{\Gamma}\colon \Gamma\to \r^3\enskip \text{is injective},
\end{equation}
where $\Gamma:=b\overline V\cap (\cup_{i=1}^\igot A_i)$ (see \eqref{eq:inT}). In particular, $b\overline V=\Gamma\cup \I(\Gamma)$ and $\Gamma\cap\I(\Gamma)=\emptyset$.

The domain $V$ and the map $Y:=H|_{\overline V}\in\Mcal_\I(\overline V)$ satisfy the conclusion of Lemma \ref{lem:compact}. Indeed, property {\rm (i)} follows from \eqref{eq:V1}, \eqref{eq:M}, and \eqref{eq:inT}; {\rm (ii)}$=$\eqref{eq:V3}; {\rm (iii)} is implied by \eqref{eq:V2}; {\rm (iv)} is given by {\rm (F.1$_{\igot\jgot}$)} and \eqref{eq:M}; {\rm (v)} follows from Claim \ref{cla:dist2} and \eqref{eq:V2}; and {\rm (vi)}={\rm (F.8$_{\igot\jgot}$)}.

This concludes the proof of Lemma \ref{lem:compact}.
\end{proof}


The following application of Lemma \ref{lem:compact} is the key tool in this section.

\begin{lemma}\label{lem:compact2}
Let $U\in\Bscr_\I(\Ncal)$, let $K\subset U$ be a compact set, and let $X\in\Mcal_\I(\overline{U})$.

Then, for any $\epsilon>0$ there exist $V\in\Bscr_\I(\Ncal)$ and $Y\in\Mcal_\I(\overline{V})$ enjoying the following properties:
\begin{enumerate}[\rm (i)]
\item $K\subset V\Subset U$.
\item If $b\overline V=\Gamma\cup \I(\Gamma)$ with $\Gamma\cap\I(\Gamma)=\emptyset$, then $Y|_{\Gamma}\colon \Gamma\to\r^3$ is an embedding.
\item $\dist_Y(K,b\overline{V})>1/\epsilon$.
\item $\|Y-X\|_{1,K}<\epsilon$.
\item $\|Y-X\|_{0,\overline V}<\epsilon$.
\item The Hausdorff distance $\delta^{\rm H}\big(Y(b\overline{V}),X(b\overline{U})\big)<\epsilon$.
\item The flux map $\pgot_Y$ of $Y$ equals the one $\pgot_X$ of $X$.
\end{enumerate}
\end{lemma}

\begin{proof}
Let $\rho_1>0$. 

Let $\{\rho_n\}_{n\in\n}$ and $\{\mu_n\}_{n\in\n}$ be the sequences of positive numbers given by
\begin{equation}\label{eq:pitag} 
\rho_n=\rho_1+\sum_{j=2}^n \frac{\asf}{j}\quad\text{and}\quad \mu_n=\sqrt{\mu_{n-1}^2+4\big(\frac{\asf}{n}\big)^2}+\frac{\asf}{n^2},\quad \forall n\geq 2,
\end{equation}
where $\asf>0$ and $\mu_1>0$ are small enough constants so that
\begin{equation}\label{eq:asf}
\mu_n <\epsilon/2\enskip \forall n\in\n. 
\end{equation}

Call $U_0:=U.$ Let $\Tsf_0$ be a metric tubular neighborhood of $b \overline U_0$ in $\Ncal$ disjoint from $K$ and denote by $\Pgot_0\colon \Tsf_0\to b\overline U_0$ the natural projection.

A standard recursive application of Lemma \ref{lem:compact} gives a sequence $\{\Xi_n=(U_n,\Tsf_n,Y_n)\}_{n\in\n}$, where $U_n\in \Bscr_\I(\Ncal)$, $\Tsf_n$ is a metric tubular neighborhood of $b\overline U_n$ in $\Ncal$, and $Y_n\in \Mcal_\I(\overline U_n)$, satisfying the following conditions:
\begin{enumerate}[{\rm (1$_n$)}]
\item If $b\overline U_n=\Gamma\cup \I(\Gamma)$ with $\Gamma\cap\I(\Gamma)=\emptyset$, then $Y_n|_\Gamma\colon \Gamma\to \r^3$ is an embedding.
\item $K\subset U_n \Subset U_{n-1} \Subset U_0$ and $\Tsf_n \subset \Tsf_{n-1}\subset \Tsf_0$, for all $n\geq 1$.  
\item $\rho_n<\dist_{Y_n}(K,b\overline U_n)$.
\item $\max\{\|X-X\circ\Pgot_0\circ \ldots \circ \Pgot_{n-1} \|_{0,b\overline U_n}\,,\,\|Y_n-X\circ\Pgot_0\circ \ldots \circ \Pgot_{n-1} \|_{0,b\overline U_n}\}<\mu_n$ for all $n\geq 1$, where $\Pgot_j\colon \Tsf_j\to b\overline U_j$ denotes the orthogonal projection.
\item $\|Y_n-X\|_{1,K}<\epsilon$ on $K$.
\item The flux map $\pgot_{Y_n}$ of $Y_n$ equals the one $\pgot_X$ of $X$.
\end{enumerate}
See \cite[Proof of Claim 4.2]{AL-Israel} for details; here we use Lemma \ref{lem:compact} instead of \cite[Lemma 3.1]{AL-Israel}.

Choose $k\in\n$ such that 
\begin{equation}\label{eq:rhok}
\rho_k>1/\epsilon,
\end{equation}
recall that $\{\rho_n\}_{n\in\n}\nearrow +\infty$. The bordered domain $V:=U_k\in\Bscr_\I(\Ncal)$ and the map $Y:=Y_k\in\Mcal_\I(\overline V)$ satisfy the conclusion of the lemma.

Indeed, {\rm (i)} is implied by {\rm (2$_k$)}; {\rm (ii)}$=${\rm (1$_k$)}; {\rm (iii)} is ensured by \eqref{eq:rhok} and {\rm (3$_k$)}; {\rm (iv)}$=${\rm (5$_k$)}; and {\rm (vii)}$=${\rm (6$_k$)}. In order to check {\rm (v)}, observe that
\begin{multline*}
\|Y-X\|_{0,b\overline V}\leq 
\\
\|Y_k-X\circ\Pgot_0\circ \ldots \circ \Pgot_{k-1}\|+\|X\circ\Pgot_0\circ \ldots \circ \Pgot_{k-1}-X\|
\stackrel{\text{\rm (4$_k$)}}{<}\mu_k+\mu_k\stackrel{\text{\rm \eqref{eq:asf}}}{<}\epsilon.
\end{multline*}
Therefore, the Maximum Principle for harmonic maps ensures that $\|Y-X\|_{0,\overline V}<\epsilon$. The same argument gives property {\rm (vi)}. 

This concludes the proof of Lemma \ref{lem:compact2}.
\end{proof}


We are now ready to prove the main result in this section. Since the proof of Theorem \ref{th:compact} relies in a standard recursive application of Lemma \ref{lem:compact2}, we will omit some of the details. We refer to the proof of \cite[Theorem 5.1]{AL-Israel} for a careful exposition.

Before going into the proof we need the following notation. For any $k\in\n,$ any compact set $K\subset\Ncal$, and any continuous injective map $f\colon K\to\r^3$, denote by
\[
\Psi(K,f,k)=\frac{1}{2k^2}\cdot\inf\left\{\|f(P)-f(Q)\|\colon P,Q\in K,\; {\rm dist}_{\sigma_\Ncal^2}(P,Q)>\frac1{k}\right\}>0,
\]
where ${\rm dist}_{\sigma_\Ncal^2}(\cdot,\cdot)$ denotes the intrinsic distance in $\Ncal$ with respect to the conformal Riemannian metric $\sigma_\Ncal$; see Remark \ref{rem:inicio}.

\begin{proof}[Proof of Theorem \ref{th:compact}]
Let $\epsilon_1$ and $\asf$ be numbers with $0<\epsilon_1<\asf/2$.

Choose an $\I$-invariant bordered domain $M_1\in\Bscr_\I(\Ncal)$ satisfying the following properties:
\begin{enumerate}[{\rm (i)}]
\item $K\subset M_1 \Subset U$. 
\item If $b\overline M_1=\Gamma\cup \I(\Gamma)$ and $\Gamma\cap\I(\Gamma)=\emptyset$, then $X|_\Gamma\colon \Gamma \to\r^3$ is an embedding.
\item The Hausdorff distance $\dgot^H(X(b\overline M_1),X(b\overline U))<\epsilon_1$.
\end{enumerate}

Denote $X_1:=X|_{\overline M_1}\in\Bscr_\I(\overline M_1)$. A standard recursive application of Lemma \ref{lem:compact2} provides a sequence $\{\Theta_n=(M_n,X_n,\Tsf_n,\epsilon_n,\tau_n)\}_{n\in\n}$, where $M_n\in \Bscr_\I(\Ncal)$, $X_n\in \Mcal_\I(\overline M_n)$, $\Tsf_n$ is a metric tubular neighborhood of $b\overline M_n$ in $\overline M_n$,  $0<\epsilon_n<\asf/2^n$, and $\tau_n>0$, for all $n\in\n$, enjoying the following properties:
\begin{enumerate}[{\rm (1$_n$)}]
\item If $\overline{\Tsf}_n=\Omega\cup\I(\Omega)$ with $\Omega\cap \I(\Omega)=\emptyset$, then   $X_n|_{\Omega}\colon \Omega \to\r^3$ is an embedding.

\item $K\subset M_{n-1}\setminus\Tsf_{n-1}\Subset M_n\setminus\Tsf_n\Subset M_n\Subset M_{n-1}\Subset M$ for all $n\geq 2$.

\item $\|X_n-X_{n-1}\|_{1,\overline M_{n-1}\setminus \Tsf_{n-1}}<\epsilon_n$ and $\|X_n-X_{n-1}\|_{0,\overline M_n}<\epsilon_n$ for all $n\geq 2$.

\item ${\rm dist}_{X_n}(K,\overline \Tsf_n)>1/\epsilon_n$ for all $n\geq 2$.

\item The Hausdorff distance $\dgot^H(X_n(b\overline M_n),X_{n-1}(b\overline M_{n-1}))<\epsilon_n$ for all $n\geq 2$.

\item There exist $a_n:=\igot\cdot E((\tau_n)^{n+1})$ points $x_{n,1},\ldots, x_{n,a_n}$ in $X_n(b\overline M_n)\subset\r^3$ such that
\[
\dgot^H(X_n(\overline \Tsf_n),\{x_{n,1},\ldots, x_{n,a_n}\})<\big(1/{\tau_n}\big)^n,
\]
where $E(\cdot)$ means integer part and $2\igot$ is the number of ends of $\Ncal$.

\item $\displaystyle\epsilon_n<\min\left\{\epsilon_{n-1}, \varrho_{n-1},\frac1{n^2(\tau_{n-1})^n}\,,\,\Psi\big( \overline{\Tsf}_{n-1}\,,\,X_{n-1}|_{\overline{\Tsf}_{n-1}}\,,\, n \big)\right\},$ where 
\[
\varrho_{n-1}=2^{-n} \min \left\{  \min_{\overline M_{k-1}\setminus\Tsf_{k-1}} \Big\| \frac{\partial X_k}{\sigma_\Ncal}\Big\|\colon k=1,\ldots,n-1 \right\}>0,\; \;n\geq 2.
\]

\item $\tau_n\geq \tau_{n-1}+1\geq n$ for all $n\geq 2$.

\item $\|{\rm dist}_{\sigma_\Ncal^2}(\,\cdot\,,b\overline M_n)\|_{0,\overline \Tsf_n}<\epsilon_n$; see Remark \ref{rem:inicio}.

\item The flux map $\pgot_{X_n}$ of $X_n$ equals the one $\pgot_X$ of $X$.
\end{enumerate}
See \cite[Proof of Claim 5.2]{AL-Israel} for details on how to construct such a sequence; here we use Lemma \ref{lem:compact2} instead of \cite[Lemma 4.1]{AL-Israel}.

Set $\Ncal_n=M_n\setminus\overline\Tsf_n$ for all $n\in \n$ and define
\[
M:=\bigcup_{n\in \n} \Ncal_n. 
\]
From (2$_n$) and (9$_n$), $n\in\n$ we obtain that
$\overline M=\cap_{n\in\n} M_n$
and the inclusion map $\Hcal_1(M,\z)\hookrightarrow\Hcal_1(\Ncal,\z)$ is a homeomorphism.

In view of {\rm (2$_n$)}, {\rm (3$_n$)}, {\rm (5$_n$)}, and {\rm (7$_n$)}, $n\in\n,$ the sequence $\{X_n|_{\overline M}\}_{n\in\n}$ uniformly converges to an $\I$-invariant continuous map 
\[
Y\colon \overline M\to\r^3
\]
such that $\max\big\{\|Y-X\|_{1,K},\|Y-X\|_{0,\overline M},\dgot^{\rm H}\big(X(b\overline{U}), Y(\Fr M)\big)\big\}<\asf$. Moreover, {\rm (3$_n$)} and {\rm (7$_n$)}, $n\in\n$, ensure that $Y|_M$ is an $\I$-invariant conformal minimal immersion, which is complete by {\rm (4$_n$)}, $n\in\n$, and its flux map $\pgot_Y$ equals the one $\pgot_X$ of $X$ by {\rm (10$_n$)}, $n\in\n$.
Finally, properties {\rm (3$_n$)} and {\rm (7$_n$)}, $n\in\n$, ensure that $(Y|_{\Fr M})^{-1}(Y(P))=\{P,\I(P)\}$ for all $P\in \Fr M$, whereas {\rm (6$_n$)}, {\rm (7$_n$)}, and {\rm (9$_n$)}, $n\in\n$, guarantee that the  Hausdorff dimension of $Y(\Fr M)$ equals to $1$, provided that $\asf$ is taken small enough from the beginning. (See \cite[Proof of Theorem 5.1]{AL-Israel} for details.)

Therefore, the domain $M$ and the map $Y$ satisfy the conclusion of Theorem \ref{th:compact}, provided that $\asf$ is chosen sufficiently small.
\end{proof}


\section{Complete non-orientable minimal surfaces properly projecting\\ into planar convex domains}\label{sec:theorem2}

Recall that we have fixed $\Ncal,$ $\I$, and $\pi$ as in Def.\ \ref{def:non} (see Remark \ref{rem:inicio}); in particular $\Ncal$ is an open Riemann surface possibly with infinite topology.

In this section we prove Theorem \ref{th:intro-proper} in the introduction.  We actually prove the following more precise result.

\begin{theorem}\label{th:proper}
Let $D\subset\r^2$ be a convex domain (possibly neither bounded nor smooth), let $U\Subset\Ncal$ be a Runge connected $\I$-invariant bordered domain, and let $X=(X_1,X_2,X_3)\in\Mcal_\I(\overline{U})$ such that
\begin{equation}\label{eq:thproper1}
(X_1,X_2)(\overline{U})\subset D.
\end{equation}
Let also $\pgot\colon\Hcal_1(\Ncal,\z)\to\r^3$ be a group morphism satisfying 
\begin{equation}\label{eq:thproper2}
\pgot|_{\Hcal_1(U,\z)}=\pgot_X\quad \text{and}\quad  \pgot(\I_*(\gamma))=-\pgot(\gamma)\quad \forall \gamma\in\Hcal_1(\Ncal,\z).
\end{equation}
(See \eqref{eq:flux-non}.)

Then for any $\epsilon>0$ there exist a Runge $\I$-invariant domain $M\subset\Ncal$ and $Y=(Y_1,Y_2,Y_3)\in\Mcal_\I(M)$ enjoying the following properties:
\begin{itemize}
\item $U\Subset M$ and $M$ is homeomorphic to $\Ncal$.
\item $\|Y-X\|_{1,\overline{U}}<\epsilon$.
\item $Y\colon M\to\r^3$ is complete.
\item $(Y_1,Y_2)(M)\subset D$ and $(Y_1,Y_2)\colon M\to D$ is a proper map.
\item The flux map $\pgot_Y$ of $Y$ equals $\pgot$.
\end{itemize}
\end{theorem}

If $D=\r^2$, the above theorem is already known; futhermore, in this particular case one can choose $M=\Ncal$ (see \cite{AL-RMnon}).


The following result contains most of the technical arguments in the proof of Theorem \ref{th:proper}. Lemma \ref{lem:compact2} will play an important role in its proof.

\begin{lemma}\label{lem:proper}
Let $\Lcal\Subset\Dcal\subset\Bgot\Subset\Agot\subset\r^2$ be bounded smooth convex domains; i.e. with smooth frontier. Let $U\Subset \widehat U\Subset\Ncal$ be Runge connected $\I$-invariant bordered domains, and let $X=(X_1,X_2,X_3)\in\Mcal_\I(\overline{U})$ such that
\begin{equation}\label{eq:lemma-proper}
(X_1,X_2)(b\overline{U})\subset\Dcal\setminus \overline \Lcal.
\end{equation}

Then, for any $\epsilon>0$ there exist a Runge $\I$-invariant bordered domain $V\Subset\Ncal$ and $Y=(Y_1,Y_2,Y_3)\in\Mcal_\I(\overline{V})$ satisfying the following properties:
\begin{enumerate}[\rm (i)]
\item $U\Subset V\Subset \widehat U$ and $V\setminus \overline U$ consists of a finite collection of open annuli.
\item $\|Y-X\|_{1,\overline U}<\epsilon$.
\item $(Y_1,Y_2)(b\overline V)\subset \Agot\setminus\overline \Bgot$.
\item $(Y_1,Y_2)(\overline V\setminus U)\subset \Agot\setminus \overline\Lcal$.
\item ${\rm dist}_Y(\overline U,b\overline V)>1/\epsilon$.
\item The flux map $\pgot_Y$ of $Y$ equals the one $\pgot_X$ of $X$.
\end{enumerate}
\end{lemma}

\begin{proof}

Let $\epsilon_0>0$.

Since $X\in\Mcal_\I(\overline U)$, there exists an $\I$-invariant bordered domain $M_0\Subset\Ncal$ such that $X\in\Mcal_\I(\overline M_0)$ and the following properties hold:
\begin{enumerate}[\rm ({A}1)]
\item $U\Subset M_0\Subset\widehat U$.
\item $M_0\setminus\overline  U$ consists of a finite collection of open annuli.
\item $(X_1,X_2)(\overline M_0\setminus U)\subset\Dcal\setminus \overline \Lcal$; see \eqref{eq:lemma-proper}.
\end{enumerate}

We now use Lemma \ref{lem:compact2} to get an $\I$-invariant bordered domain $M\Subset\Ncal$ and an $\I$-invariant conformal minimal immersion $F=(F_1,F_2,F_3)\in\Mcal_\I(\overline M)$, such that:
\begin{enumerate}[\rm ({B}1)]
\item $U\Subset M\Subset M_0\Subset\widehat U$; see {\rm (A1)}.
\item $M\setminus \overline U$ consists of a finite collection of open annuli; see {\rm (A2)}.
\item $\|F-X\|_{1,\overline U}<\epsilon_0$.
\item $(F_1,F_2)(\overline M\setminus U)\subset\Dcal\setminus \overline \Lcal$; see {\rm (A3)} and Lemma \ref{lem:compact2} {\rm (v)}.
\item ${\rm dist}_F(\overline U,b\overline M)>1/\epsilon_0$.
\item The flux map $\pgot_F$ of $F$ equals the one $\pgot_X$ of $X$.
\end{enumerate}

Write
\[
b\overline M=\cup_{i=1}^\igot (\alpha_i\cup\I(\alpha_i));
\]
where $\{\alpha_i\}_{i=1}^\igot$ are pairwise disjoint smooth Jordan curves with $\alpha_i\cap \I(\alpha_i)=\emptyset$. Denote $\alpha=\cup_{i=1}^\igot \alpha_i$. It follows that $b\overline M=\alpha\cup\I(\alpha)$ and $\alpha\cap\I(\alpha)=\emptyset$.

Since $\Lcal$ is convex, {\rm (B4)} ensures that for any $P\in \alpha$ there exist a line $l_P$ in $\r^2$  and an open neighborhood $O_P$ of $P$ in $\alpha$ such that
\begin{equation}\label{eq:2OP}
\big((F_1,F_2)(Q)+l_P\big)\cap \overline\Lcal=\emptyset\quad \forall Q\in O_P.
\end{equation}

Since $\alpha$ is compact, then there exist $\jgot\in\n,$ $\jgot\geq 3,$ and a family of compact Jordan arcs $\{\alpha_{i,j}\colon (i,j)\in I=\{1,\dots,\igot\}\times\z_\jgot\}$ meeting the following requirements:
\begin{enumerate}[{\rm ({C}1)}]
\item $\cup_{j\in\z_\jgot} \alpha_{i,j}=\alpha_i$ for all $i\in\{1,\dots,\igot\}$.
\item $\alpha_{i,j}$ and $\alpha_{i,j+1}$ have a common endpoint $Q_{i,j}$ and are otherwise disjoint for all $(i,j)\in I$.
\item $\alpha_{i,j}\cap \alpha_{i,k}=\emptyset$ for all $(i,j)\in I$ and $k\in \z_\jgot\setminus\{j,j+1\}$.
\item $\alpha_{i,j}\subset O_{R_{i,j}}$ for a point $R_{i,j}\in\alpha$, for all $(i,j)\in I$.
\end{enumerate}

For any $(i,j)\in I$ label $l_{i,j}:=l_{R_{i,j}}$, $O_{i,j}:=O_{R_{i,j}}$, and denote by $u_{i,j}$ the unitary vector in $\r^2$ orthogonal to $l_{i,j}$ and pointing to the connected component of $\r^2\setminus ((F_1,F_2)(R_{i,j})+l_{i,j})$ disjoint from $\overline \Lcal$. Set 
\[
e_{i,j}^3:=(u_{i,j},0)\in\s^2
\]
and denote by $\pi_{i,j}\colon \r^2\to {\rm span}\{u_{i,j}\}\subset \r^2$ the orthogonal projection.

For any $(i,j)\in I$ choose $\{e_{i,j}^1,e_{i,j}^2\}\subset\r^3$ such that $\{e_{i,j}^1,e_{i,j}^2,e_{i,j}^3\}$ is an orthonormal basis of $\r^3$, and denote
\begin{equation}\label{eq:2Bij}
B_{i,j}=\Big({e_{i,j}^1}^T\,,\,{e_{i,j}^2}^T\,,\,{e_{i,j}^3}^T\Big)^{-1}.
\end{equation}

Let $\{r_{i,j}\colon (i,j)\in I\}$ be a family of pairwise disjoint analytical compact Jordan arcs in $M_0\setminus M$ meeting the following requirements:
\begin{enumerate}[\rm ({D}1)]
\item $r_{i,j}\subset O_{i,j}\cap O_{i,j+1}$ for all $(i,j)\in I$.
\item $r_{i,j}$ has initial point $Q_{i,j}$ and is otherwise disjoint from $b\overline M$  for all $(i,j)\in I$. Denote by $P_{i,j}$ the other endpoint of $r_{i,j}$ for all $(i,j)\in I$.
\item The set 
\[
S:=\overline M \cup\big(\cup_{(i,j)\in I} (r_{i,j}\cup\I(r_{i,j}))\big)\subset M_0\subset \Ncal
\]
is $\I$-admissible in the sense of Def.\ \ref{def:admi}. 
\end{enumerate} 
See Fig.\ \ref{fig:proper}.

We first deform $F$ over the arcs $r_{i,j}$ and $\I(r_{i,j})$, $(i,j)\in I$.

Extend $F|_{\overline M}$ to an $\I$-invariant generalized minimal immersion  $\widehat F=(\widehat F_1,\widehat F_2,\widehat F_3)\in \Mcal_{\ggot,\I}(S)$ such that (see Def.\ \ref{def:gene-min})
\begin{enumerate}[\rm ({E}1)]
\item $\pi_{i,a}\big((\widehat F_1,\widehat F_2)(r_{i,j})\big)\cap \pi_{i,a}(\overline \Lcal)=\emptyset$ for all $(i,j)\in I$ and $a\in\{j,j+1\}$.
\item $\pi_{i,a}\big((\widehat F_1,\widehat F_2)(P_{i,j})\big)\notin \pi_{i,a}(\overline \Agot)=\emptyset$ for all $(i,j)\in I$ and $a\in\{j,j+1\}$.
\end{enumerate}
For instance, one can take $(\widehat F_1,\widehat F_2)(r_{i,j})$ to be $\Ccal^0$ close to a long enough straight segment in $\r^2$ with initial point $(F_1,F_2)(Q_{i,j})$ and directed by a vector $\widehat u_{i,j}\in \r^2$ with $\langle\widehat u_{i,j},u_{i,a}\rangle>0$ for $a=j,j+1$. Such a vector $\widehat u_{i,j}$ exists since $u_{i,j}\neq -u_{i,j+1}$; take into account that $Q_{i,j}\in O_{i,j}\cap O_{i,j+1}$ by {\rm (C4)}.

Applying Theorem \ref{th:Mergelyan} to any marked immersion $\widehat F_\varpi=(\widehat F,\varpi)\in\Mcal_{\ggot,\I}^*(S)$ and $\pgot=\pgot_{\widehat F_\varpi}=\pgot_X\colon\Hcal_1(\Ncal,\z)\to\r^3$ (see Def.\ \ref{def:marked} and {\rm (B6)}), one obtains an $\I$-invariant conformal minimal immersion $H=(H_1,H_2,H_3)\in\Mcal_\I(\overline M_0)$ satisfying:
\begin{enumerate}[\rm ({F}1)]
\item $\|H-F\|_{1,\overline M}<\epsilon_0$.

\item $\|H-\widehat F\|_{0,S}<\epsilon_0$.

\item $(H_1,H_2)(\overline M\setminus U)\subset\Dcal\setminus \overline \Lcal$; see {\rm (B4)}.

\item ${\rm dist}_H(\overline U,b\overline M)>1/\epsilon_0$; see {\rm (B5)}.

\item $\pi_{i,j}\big((H_1,H_2)(r_{i,j-1}\cup\alpha_{i,j}\cup r_{i,j})\big)\cap \pi_{i,j}(\overline \Lcal)=\emptyset$ for all $(i,j)\in I$; see {\rm (C4)}, \eqref{eq:2OP}, and {\rm (E1)}.

\item $\pi_{i,j}\big((H_1,H_2)(\{P_{i,j-1},P_{i,j}\})\big)\cap \pi_{i,j}(\overline \Agot)=\emptyset$ for all $(i,j)\in I$; see {\rm (E2)}.

\item The flux map $\pgot_H$ of $H$ equals the one $\pgot_X$ of $X$.
\end{enumerate}

By continuity of $H$, there exists an $\I$-invariant bordered domain $W\subset \Ncal$ satisfying the following properties:
\begin{enumerate}[\rm ({G}1)]
\item $M\Subset W\Subset M_0$.

\item $S\subset \overline W $ and $S\cap b\overline W =\cup_{(i,j)\in I} \{P_{i,j},\I(P_{i,j})\}$.

\item $W \setminus\overline M$ consists of a finite collection of open annuli.

\item $\pi_{i,j}\big((H_1,H_2)(\Omega_{i,j})\big)\cap \pi_{i,j}(\overline \Lcal)=\emptyset$, where $\Omega_{i,j}$ is the closure of the connected component of $W\setminus S$ bounded by $r_{i,j-1}$, $\alpha_{i,j}$, $r_{i,j}$, and the piece $\beta_{i,j}$ of $b\overline W$ which connects $P_{i,j-1}$ and $P_{i,j}$ and is otherwise disjoint from $S$; see {\rm (F5)}.

\item $\pi_{i,j}\big((H_1,H_2)(\overline{\beta_{i,j}\setminus \Gamma_{i,j}})\big)\cap\pi_{i,j}(\overline \Agot)=\emptyset$, where $\Gamma_{i,j}$ is a compact sub-arc of $\beta_{i,j}\setminus\{P_{i,j-1},P_{i,j}\}$; see {\rm (F6)}.

\item $(\partial H B_{i,j}^T)_3$ vanishes nowhere on $b\overline W$ for all $(i,j)\in I$; here $(\cdot)_3$ denotes third coordinate in $\r^3$ and  $B_{i,j}$ is given by \eqref{eq:2Bij}.
\end{enumerate}
See Fig.\ \ref{fig:proper}.
\begin{figure}[ht]
    \begin{center}
    \scalebox{0.32}{\includegraphics{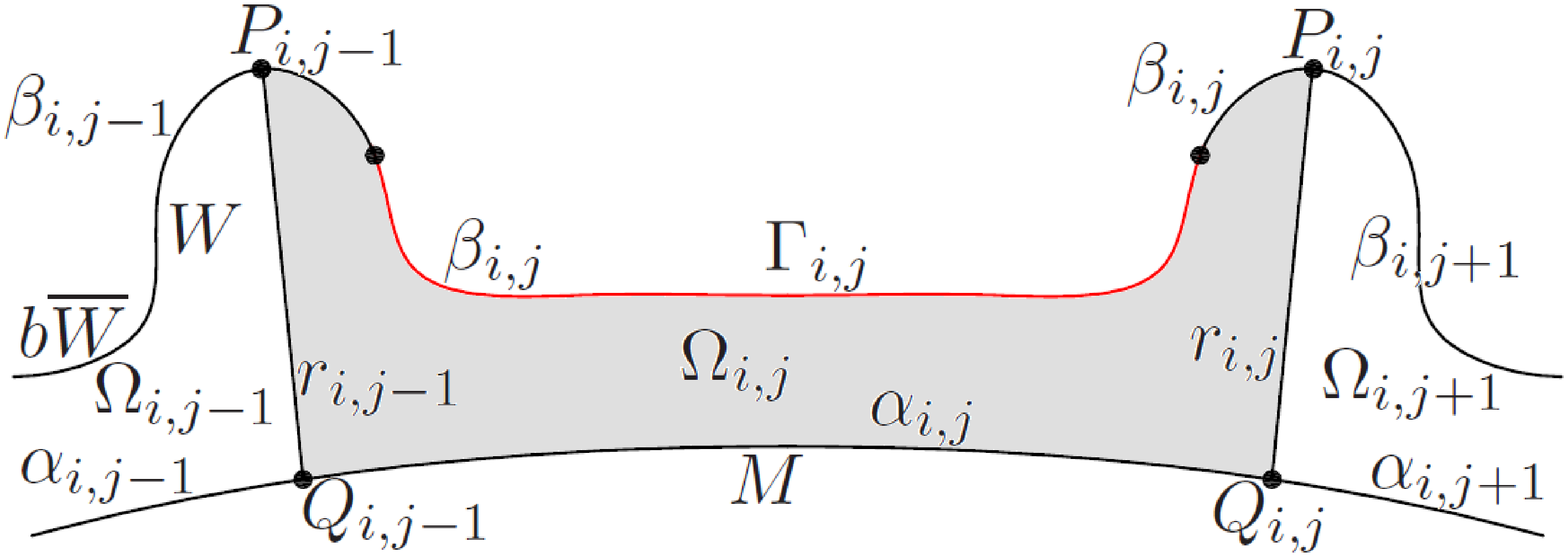}}
        \end{center}
\caption{$\overline W\setminus M$}
\label{fig:proper}
\end{figure}

We now deform $H$ over the disc $\Omega_{i,j}$ and $\I(\Omega_{i,j})$ having particular control over the arcs $\Gamma_{i,j}$ and $\I(\Gamma_{i,j})$, $(i,j)\in I$. 

Let $\eta\colon \{1,\ldots,\igot\jgot\}\to\{1,\ldots,\igot\}\times \z_\jgot$ be the bijection $\eta(k)=(E(\frac{k-1}{\jgot})+1,k-1),$ where $E(\cdot)$ means integer part. 

We now recursively construct a sequence $\{H^0=H, H^1,\ldots, H^{\igot\jgot}\}\subset \Mcal_\I(\overline M_0)$, $H^k=(H^k_1,H^k_2,H^k_3)$ for all $k\in\{0,1,\ldots,\igot\jgot\}$, satisfying the following properties for all $k\in \{0,1,\ldots,\igot\jgot\}$:
\begin{enumerate}[\rm (1$_{k}$)]
\item $\|H^k-F\|_{1,\overline M}<\epsilon_0$.


\item $\|H^k-H^{k-1}\|_{1,\overline{W\setminus (\Omega_{\eta(k)}\cup \I(\Omega_{\eta(k)}))}} <\epsilon_0/\igot\jgot$, $k\geq 1$.

\item $\langle H^k-H^{k-1},e_{\eta(k)}^3 \rangle=0$, $k\geq 1$.

\item ${\rm dist}_{H^k}(\overline U,b\overline M)>1/\epsilon_0$.

\item $\pi_{\eta(a)}\big((H^k_1,H^k_2)(\Omega_{\eta(a)})\big)\cap \pi_{\eta(a)}(\overline \Lcal)=\emptyset$ for all $a\in\{1,\ldots,\igot\jgot\}$.

\item $\pi_{\eta(a)}\big((H^k_1,H^k_2)(\overline{\beta_{\eta(a)}\setminus \Gamma_{\eta(a)}})\big)\cap\pi_{\eta(a)}(\overline \Agot)=\emptyset$ for all $a\in\{1,\ldots,\igot\jgot\}$.

\item $(H^k_1,H^k_2)(\Gamma_{\eta(a)})\cap\overline\Agot=\emptyset$ for all $a\in\{1,\ldots,k\}$, $k\geq 1$.

\item $(\partial (H^k B_{i,j}^T))_3$ vanishes nowhere on $b\overline W$ for all $(i,j)\in I$; see \eqref{eq:2Bij}.

\item The flux map $\pgot_{F_k}$ of $F_k$ equals the one $\pgot_X$ of $X$.

\item $(H^k_1,H^k_2)(\overline M\setminus U)\subset\Dcal\setminus \overline \Lcal$.
\end{enumerate}
Indeed, observe that {\rm (1$_0$)}={\rm (F1)}, {\rm (4$_0$)}={\rm (F4)}, {\rm (5$_0$)}={\rm (G4)}, {\rm (6$_0$)}={\rm (G5)}, {\rm (8$_0$)}={\rm (G6)}, {\rm (9$_0$)}={\rm (F7)}, and {\rm (10$_0$)}={\rm (F3)}, whereas {\rm (2$_0$)}, {\rm (3$_0$)}, and {\rm (7$_0$)} make no sense. Reason by induction and assume that we already have $H^0,\ldots, H^{k-1}$, for some $k\in\{1,\ldots,\igot\jgot\}$, satisfying the corresponding properties. Let us construct $H^k$.

Denote $G=(G_1,G_2,G_3):=H^{k-1} B_{\eta(k)}^T\in \Mcal_\I(\overline M_0)$, where $B_{\eta(k)}$ is the orthogonal matrix \eqref{eq:2Bij}. It follows that $G\in\Mcal_\I(\overline M_0)$ and 
\[
G_3=\langle H^{k-1}, e_{\eta(k)}^3\rangle.
\]
Denote
\[
S_k:= \overline M \cup \Big(\Gamma_{\eta(k)}\cup \I(\Gamma_{\eta(k)})\Big)\cup \Big(\bigcup_{a\neq k} \Omega_{\eta(a)}\cup \I(\Omega_{\eta(a)})\Big)
\]
and observe that $S_k$ is $\I$-admissible (Def.\ \ref{def:admi}). Extend $G|_{S_k\setminus (\Gamma_{\eta(k)}\cup \I(\Gamma_{\eta(k)}))}$ to an $\I$-invariant generalized minimal immersion $\widehat G=(\widehat G_1,\widehat G_2,\widehat G_3)\in\Mcal_{\ggot,\I}(S_k)$ (cf.\ Def. \ref{def:gene-min}) such that
$
\widehat G_3= G_3|_{S_k}
$
and
\[
\pi_{\eta(k)}\left(\Big( \big(\widehat G (B_{\eta(k)}^T)^{-1}\big)_1,\big(\widehat G (B_{\eta(k)}^T)^{-1}\big)_2\Big)(\Gamma_{\eta(k)})\right)\cap \overline\Agot=\emptyset,
\]
where $(\cdot)_j$ means $j$-th coordinate in $\r^3$.
For instance, choose $\widehat G|_{\Gamma_k}=G|_{\Gamma_k}+(x_0,y_0,0)$ where $(x_0,y_0)\in\r^2$ is a constant with large enough norm.

In view of {\rm (8$_{k-1}$)}, one can apply Theorem \ref{th:Mergelyan} {\rm (II)} to any marked immersion $\widehat G_\varpi=(\widehat G,\varpi)\in\Mcal_{\ggot,\I}^*(S_k)$ such that $(\partial \widehat G_\varpi)_3=(\partial G_3)|_{S_k}$ and  $\pgot=\pgot_{\widehat G_\varpi}=\pgot_G\colon\Hcal_1(\Ncal,\z)\to\r^3$, obtaining an $\I$-invariant conformal minimal immersion $\widetilde G=(\widetilde G_1,\widetilde G_2,\widetilde G_3)\in\Mcal_\I(\overline M_0)$ such that 
$\widetilde G_3=G_3$
and the flux map $\pgot_{\widetilde G}$ of $\widetilde G$ equals the one $\pgot_G$ of $G$.
It is now straightforward to check that, if the approximation of $\widehat G$ by $\widetilde G$ in $S_k$ is close enough, then $H^k:=\widetilde G (B_{\eta(k)}^T)^{-1}\in\Mcal_\I(\overline M_0)$ satisfies properties {\rm (1$_k$)}--{\rm (10$_k$)}. This concludes the construction of the sequence $H^k,\ldots,H^{\igot\jgot}\in\Mcal_\I(\overline M_0)$.

From {\rm (6$_{\igot\jgot}$)} and {\rm (7$_{\igot\jgot}$)}, we obtain that $(H^{\igot\jgot}_1, H^{\igot\jgot}_2)(b\overline W)\cap \overline \Agot=\emptyset$, whereas {\rm (10$_{\igot\jgot}$)} and the Convex Hull property of minimal surfaces ensure that $(H^{\igot\jgot}_1, H^{\igot\jgot}_2)(\overline M)\subset\Dcal$. Since $\Dcal\Subset \Agot$ and $H^{\igot\jgot}$ is $\I$-invariant, then there exists an $\I$-invariant bordered domain $V\Subset \Ncal$ such that
\begin{equation}\label{eq:2V}
M\Subset V\Subset W, \quad \text{$V\setminus \overline M$ consists of $\igot$ open annuli,} 
\end{equation}
and
\begin{equation}\label{eq:2V2}
(H^{\igot\jgot}_1, H^{\igot\jgot}_2)(b\overline V)\subset \Agot\setminus \overline \Bgot.
\end{equation}
In particular, by the Convex Hull property of minimal surfaces,
\begin{equation}\label{eq:2V3}
(H^{\igot\jgot}_1, H^{\igot\jgot}_2)(\overline V)\subset \Agot.
\end{equation}

Set $Y:=H^{\igot\jgot}|_{\overline V}\in\Mcal_\I(\overline V)$ and notice that $Y$ and $V$ satisfy the conclusion of the lemma provided that $\epsilon_0$ is chosen small enough from the beginning. Indeed, Lemma \ref{lem:proper} {\rm (i)} follows from  {\rm (B1)}, {\rm (B2)}, and \eqref{eq:2V}; {\rm (ii)} is implied by {\rm (B1)}, {\rm (B3)}, and {\rm (1$_{\igot\jgot}$)}; {\rm (iii)}=\eqref{eq:2V2}; {\rm (iv)} is ensured by \eqref{eq:2V3}, {\rm (10$_{\igot\jgot}$)}, {\rm (5$_{\igot\jgot}$)}, and the fact that $\overline V\setminus \overline M\subset W\setminus \overline M= \cup_{a=1}^{\igot\jgot} \big(\Omega_{\eta(a)}\cup\I(\Omega_{\eta(a)})\big)$; {\rm (v)} is guaranteed by {\rm (4$_{\igot\jgot}$)} and \eqref{eq:2V}; and {\rm (vi)}={\rm (9$_{\igot\jgot}$)}.

This proves the lemma.
\end{proof}

We can now prove the main result in this section. It will follow from a standard recursive application of Lemma \ref{lem:proper}.

\begin{proof}[Proof of Theorem \ref{th:proper}]
Let $\epsilon_0>0$.

By \eqref{eq:thproper1} and Theorem \ref{th:mink}, we can take a sequence $\{\Dcal^j\}_{j\in\n\cup\{0\}}$ of bounded smooth convex domains in $\r^2$, such that $\Dcal^{j-1}\Subset\Dcal^j\Subset D$ for all $j\in\n$,
\begin{equation}\label{eq:D1}
\big\{\overline \Dcal^j\}\nearrow \overline D,
\end{equation}
and
\begin{equation}\label{eq:D2}
(X_1,X_2)(b\overline U)\subset \Dcal^1\setminus \overline\Dcal^0.
\end{equation}

Call $U_0:=U$, and take also a sequence $\{U_j\}_{j\in\n}$ of Runge connected $\I$-invariant bordered domains in $\Ncal$, satisfying:
\begin{enumerate}[\rm (a)]
\item $U_{j-1}\Subset U_j$ for all $j\in\n$.
\item The Euler characteristic $\chi(U_j\setminus \overline U_{j-1})\in\{0,-2\}$ for all $j\in \n$.
\item $\Ncal=\cup_{j\in\n} U_j$.
\end{enumerate}
Such a sequence is constructed in \cite[Remark 5.8]{AL-RMnon}. 

Call $M_0:=U_0$ and $Y^0=(Y_1^0,Y_2^0,Y_3^0):=X$. Let us construct a sequence $\{(\epsilon_j,M_j,Y^j)\}_{j\in\n}$, where $\epsilon_j>0$, $M_j$ is a Runge connected $\I$-invariant bordered domain in $\Ncal$, and $Y^j=(Y_1^j,Y_2^j,Y_3^j)\in\Mcal_\I(\overline M_j)$ enjoy the following properties:
\begin{enumerate}[\rm (1$_j$)]
\item $M_{j-1}\Subset M_j\Subset U_j$ for all $j\in\n$.
\item The inclusion map $M_j \hookrightarrow U_j$ induces an isomorphism $\Hcal_1(M_j,\z)\to \Hcal_1(U_j,\z)$ for all $j\in\n$.
\item $\|Y^j-Y^{j-1}\|_{1,M_{j-1}}<\epsilon_j<\epsilon_0/2^j$ for all $j\in\n$.
\item $(Y_1^j,Y_2^j)(b\overline M_j)\subset \Dcal^{j+1}\setminus\overline\Dcal^j$ for all $j\in\n$.
\item $(Y_1^j,Y_2^j)(\overline M_j\setminus M_{j-1})\subset \Dcal^j\setminus\overline\Dcal^{j-1}$ for all $j\in\n$.
\item ${\rm dist}_{Y^j}(\overline M_{j-1},b\overline M_j)>1/\epsilon_j>2^j/\epsilon_0$.
\item The flux map $\pgot_{Y^j}$ of $Y^j$ equals $\pgot|_{\Hcal_1(M_j,\z)}$.
\end{enumerate}

Indeed, observe that $M_0$ and $Y^0$ satisfy {\rm (2$_0$)}, {\rm (4$_0$)}, and {\rm (7$_0$)} by the fact $(M_0,Y^0)=(U_0,X)$, \eqref{eq:D2}, and \eqref{eq:thproper2}, respectively, whereas the remaining requirements make no sense for $j=0$. Reason by induction and assume that we already have $(\epsilon_{j-1},M_{j-1},Y^{j-1})$ for some $j\in\n$, satisfying the corresponding properties, and let us construct $(\epsilon_j,M^j,Y^j)$.

Choose $\epsilon_j<\epsilon_0/2^j$.

If the Euler characteristic $\chi(U_j\setminus \overline U_{j-1})=0$, then we directly obtain $M_j$ and $Y^j$ as the resulting data to apply Lemma \ref{lem:proper} to
\[
\big( \Lcal \,,\, \Dcal \,,\, \Bgot \,,\, \Agot \,,\, U \,,\, \widehat U \,,\, X \,,\, \epsilon \big) = \big( \Dcal^{j-1} \,,\, \Dcal^j \,,\, \Dcal^j \,,\, \Dcal^{j+1} \,,\, M_{j-1} \,,\, U_j \,,\, Y^{j-1} \,,\, \epsilon_j \big).
\]

Otherwise, properties {\rm (b)} and {\rm (2$_{j-1}$)} ensure that the Euler characteristic $\chi(U_j\setminus \overline U_{j-1})=\chi(U_j\setminus \overline M_{j-1})=-2$. In this case, by elementary topological arguments, there exists a compact Jordan arc $\gamma\subset U_j$ such that:
\begin{itemize}
\item $\gamma$ has its endpoints in $b\overline M_{j-1}$ and is otherwise disjoint from $\overline M_{j-1}$.
\item $\gamma\cap\I(\gamma)=\emptyset$.
\item $S:=\overline M_{j-1}\cup\gamma\cup\I(\gamma)$ is an $\I$-admissible subset in $\Ncal$ (see Def.\ \ref{def:admi}).
\item The Euler characteristic $\chi(U_j\setminus S_j)=0$.
\end{itemize}
(Cf.\ \cite[Remark 5.8]{AL-RMnon}.)

Extend $Y^{j-1}$ to an $\I$-invariant generalized minimal immersion $\widehat Y=(\widehat Y_1,\widehat Y_2,\widehat Y_3)\in\Mcal_{\ggot,\I}(S)$ satisfying $\widehat Y(\gamma)\subset \Dcal^j\setminus\overline \Dcal^{j-1}$; existence of such extension is guaranteed by {\rm (4$_{j-1}$)}. Applying Theorem \ref{th:Mergelyan} {\rm (I)} to any marked immersion $\widehat Y_\varpi=(\widehat Y,\varpi)\in\Mcal_{\ggot,\I}^*(S)$, such that the generalized flux map $\pgot_{\widehat Y_\varpi}$ of $\widehat Y_\varpi$ equals $\pgot|_{\Hcal_1(U_j,\z)}$ (see \eqref{eq:fluxgene}), we obtain an $\I$-invariant bordered domain $W$ and $F=(F_1,F_2,F_3)\in\Mcal_\I(\overline W)$ meeting the following requirements:
\begin{itemize}
\item $\overline M_{j-1}\subset S\subset W\Subset U_j$ and the Euler characteristic $\chi(U_j\setminus\overline W)=0$.
\item $\|F-Y^{j-1}\|_{1,\overline M_{j-1}}\approx 0$.
\item $(F_1,F_2)(\overline W\setminus M_{j-1})\subset \Dcal^j\setminus\overline \Dcal^{j-1}$.
\item The flux map $\pgot_F$ of $F$ equals $\pgot|_{\Hcal_1(U_j,\z)}$.
\end{itemize}

This reduces the construction of the triple $(\epsilon_j,M_j,Y^j)$ to the already done in the case when $\chi(U_j\setminus\overline U_{j-1})=0$, concluding the construction of the desired sequence $\{(\epsilon_j,M_j,Y^j)\}_{j\in\n}$.

Set
\[
M:=\bigcup_{j\in\n} M_j\subset\Ncal,
\]
which is a Runge $\I$-invariant domain, homeomorphic to $\Ncal$, satisfying $U\Subset M$; take into account properties {\rm (1$_j$)} and {\rm (2$_j$)}, $j\in\n$. By {\rm (3$_j$)}, $j\in\n$, the sequence $\{Y^j\}_{j\in\n}$ converges uniformly in compact subsets of $M$ to an $\I$-invariant conformal minimal immersion $Y=(Y_1,Y_2,Y_3)\in\Mcal_\I(M)$ with $\|Y-X\|_{1,\overline U}<\epsilon$, provided that $\epsilon_j$ is chosen small enough for each $j\in\n$. Furthermore, {\rm (6$_j$)}, $j\in\n$, ensure that $Y$ is complete, whereas \eqref{eq:D1} and {\rm (5$_j$)}, $j\in\n$, imply that $(Y_1,Y_2)(M)\subset D$ and $(Y_1,Y_2)\colon M\to D$ is a proper map. Finally, the flux map $\pgot_Y$ of $Y$ equals $\pgot$ by {\rm (7$_j$)}, $j\in\n$.

This concludes the proof.
\end{proof}



\end{document}